\documentclass[12pt,dvipsnames,reqno]{amsart}
\usepackage{amsmath, amsthm, amscd, amsfonts, amssymb, mathrsfs, graphicx, mathabx, bbm, tikz, caption, enumitem, mathtools, comment, pgfplots, multicol}
    \pgfplotsset{compat=1.16}
    \usepgfplotslibrary{fillbetween}
    \makeatletter
    \def\l@subsection{\@tocline{2}{0pt}{2.9pc}{5pc}{}}
    \def\l@subsubsection{\@tocline{2}{0pt}{5pc}{7.5pc}{}}
    \makeatother
\usepackage[margin=2.5cm]{geometry}
    \setlist[enumerate,1]{ labelindent=3pt, leftmargin=*, labelsep=0.5em }
    
    \allowdisplaybreaks
\usepackage[alphabetic]{amsrefs}
\usepackage[bookmarksnumbered, colorlinks, plainpages]{hyperref}
    \hypersetup{colorlinks=true, linkcolor=MidnightBlue, citecolor=BrickRed, filecolor=MidnightBlue, urlcolor=MidnightBlue}
    \mathtoolsset{showonlyrefs,showmanualtags} 
    
\numberwithin{equation}{section}
\newtheorem{mainthm}{Theorem}
    
\newtheorem{theorem}[equation]{Theorem}
\newtheorem{lemma}[equation]{Lemma}
\newtheorem{proposition}[equation]{Proposition}
\newtheorem{corollary}[equation]{Corollary}
\theoremstyle{definition}
\newtheorem{definition}[equation]{Definition}

\newtheorem{question}[equation]{Question}

\theoremstyle{remark}

\newcommand{\R}{\mathbb{R}}
\newcommand{\N}{\mathcal{N}}

\newcommand{\Z}{\mathbb Z}

\newcommand{\one}{\mathbbm{1}}

\DeclareMathOperator{\lcm}{lcm}

\newcommand{\E}{ \mathop{ \mathchoice {\vcenter{\hbox{\LARGE $\mathbb E$}}} {\mathbb E} {\mathbb E} {\mathbb E} }\displaylimits }

\begin{document}
\title{Ergodic Averages over Shrinking Sectors of $\mathbb Z[i]$}
\begin{abstract}
We study ergodic averages over Gaussian integers in shrinking angular sectors, which impose an increasingly restrictive directional relation between the real and imaginary coordinates. In uniquely ergodic systems, averages of continuous functions along $\Omega$ converge to the invariant integral uniformly in the base point and sector location for sectors whose widths shrink subpolynomially in the norm cutoff.

For finitely generated multiplicative actions, we prove a quantitative sector--disk comparison in an explicit logarithmic shrinking range, with a power-saving error in $\log N$, uniformly in the base point and sector location. The exponents depend only on the number of distinct prime transformations, and the result requires neither ergodicity nor angular-distribution hypotheses on the sets of primes inducing the individual transformations. Strong unique ergodicity then yields convergence to the invariant integral at each fixed base point.

Consequences include asymptotic independence of relative angular position and $T^{\Omega(n)} x$, Gaussian Liouville cancellation, and joint equidistribution modulo integers of prime-factor counts associated with two prime classes having divergent reciprocal-norm sums.
\end{abstract}

\author[A. Burgin]{Alex Burgin}
\address{School of Mathematics, Georgia Institute of Technology, Atlanta GA 30332, USA}
\email{aburgin6@gatech.edu}
\thanks{A.B. is supported in part by the Simons Foundation International through a Simons Dissertation Fellowship [SFI-MPS-SDF-00026777], the U.S. Department of Education through a GAANN fellowship, and the National Science Foundation [DMS-2247254]}

\author[C. Giannitsi]{Christina Giannitsi}
\address{Department of Mathematics, Virginia Tech, Blacksburg, VA, 24060, USA}
\email{cgiannitsi@vt.edu}

\maketitle
\tableofcontents

\section{Introduction}
We study whether arithmetic ergodic averages over the Gaussian integers retain their limiting behavior under restriction to shrinking angular sectors. Geometrically, such restrictions impose an increasingly strong coupling between the real and imaginary coordinates: while the radial scale tends to infinity, the admissible lattice points are confined to an increasingly narrow range of directions. For finitely generated, strongly uniquely ergodic multiplicative actions, Donoso, Le, Moreira, and Sun \cite{DLMS2024}*{Theorem E} established convergence along dilates of a fixed planar region, including disks and fixed sectors. We allow the angular width to decrease with the norm cutoff, so that the averaging regions no longer arise as dilates of a fixed region.

The essential difficulty is the shrinking width, rather than the changing initial angle. When the width stays bounded away from zero, fixed-sector convergence can be transferred to moving sectors by a finite angular approximation. Shrinking sectors, however, occupy a vanishing proportion of the disk: an error negligible on the disk scale may therefore be significant after normalization by the sector cardinality. In particular, the fixed-region convergence theorem of Donoso, Le, Moreira, and Sun does not apply directly to a family of sectors whose widths shrink with $N$, since the underlying planar region is then itself changing with $N$. Moreover, because a sector of width $\gamma$ contains only a proportion comparable to $\gamma$ of the disk, estimates that are negligible on the disk scale need not remain negligible after normalization by the sector cardinality. Thus the problem requires control at the shrinking angular scale itself.

Our main results have complementary scopes. For $\Omega$-averages in uniquely ergodic systems, we obtain uniform convergence throughout the subpolynomial shrinking range. For arbitrary finitely generated multiplicative actions, we instead prove a quantitative sector-disk comparison in a logarithmic shrinking range, requiring neither ergodicity nor an angular-distribution hypothesis on the sets of primes inducing the individual transformations. As a consequence of the first result, in the same subpolynomial shrinking range, the relative angular position within the sector and the dynamical variable $T^{\Omega(n)} x$ are asymptotically independent.

We write $\Omega(n)$ for the number of Gaussian prime factors of $n$, counted with multiplicity, and $\mathcal{N}(n) := |n|^2$. Let $K_{s,\gamma}(N)$ denote the nonzero Gaussian integers of norm at most $N$ whose arguments lie in $[s, s + \gamma)$, interpreted modulo $2\pi$. We use $\E$ for uniform averaging. Further notation and conventions are collected at the end of the introduction.

\begin{mainthm}\label{thm:additivesec}
Let $X$ be a compact metric space, let $T : X \rightarrow X$ be continuous, and suppose that $\mu$ is the unique $T$-invariant Borel probability measure on $X$. Let $(\gamma_N)_{N \geq 3}$ be a sequence of positive real numbers satisfying $\gamma_N \rightarrow 0$ and $\gamma_N^{-1} = N^{o(1)}$ as $N \rightarrow \infty$. Then, for every $f \in C(X)$,
\begin{align}
\sup_{\gamma_N \leq \gamma \leq 2\pi} 
\sup_{s \in [0, 2\pi)} 
\sup_{x \in X} 
\left|\E_{n \in K_{s,\gamma}(N)} f(T^{\Omega(n)} x) 
- \int_X f \, d\mu\right|
& \longrightarrow 0,
\qquad N \rightarrow \infty.
\end{align}
\end{mainthm}

Theorem \ref{thm:additivesec} relies on the following uniform arithmetic shift invariance theorem. For a finite nonempty set $E \subset \Z[i] \setminus \{0\}$ and $k \in \mathbb{N}_0$, write $p(E, k) := |\{n \in E : \Omega(n) = k\}|/|E|$ for the proportion of elements having exactly $k$ prime factors, counted with multiplicity.

\begin{mainthm}\label{thm:omegaavg}
Let $(\gamma_N)_{N \geq 3}$ be a sequence of positive real numbers satisfying $\gamma_N \rightarrow 0$ and $\gamma_N^{-1} = N^{o(1)}$ as $N \rightarrow \infty$. Then
\begin{align}
\sup_{\gamma_N \leq \gamma \leq 2\pi} \sup_{s \in [0, 2\pi)} \sum_{k \geq 1} \left|p(K_{s,\gamma}(N), k) - p(K_{s,\gamma}(N), k + 1)\right|
& \longrightarrow 0,
\qquad 
N \rightarrow \infty.
\end{align}
\end{mainthm}

Up to the negligible contribution of the units, Theorem \ref{thm:omegaavg} gives asymptotic $\ell^1$-invariance of the distribution of $\Omega$ under translation by $1$, uniformly over the stated sectors; by iteration, the same holds for every fixed integer shift. This information is not supplied by normal-order or central-limit estimates alone and is precisely what yields invariance of the dynamical weak-$*$ limits.

To formulate the multiplicative results, write $\mathcal{P}_{\Z[i]}$ for the set of all Gaussian prime elements, including all associates. For $N \geq 1$, let $B_N$ denote the nonzero Gaussian integers of norm at most $N$. We first introduce the notions of finite generation, pretended invariance, and strong unique ergodicity.
\begin{definition}
\begin{enumerate}[label=(\alph*)]
\item A multiplicative topological dynamical system is a pair $(Y, S)$, where $Y$ is a compact metric space and $S = (S_n)_{n \in \Z[i] \setminus \{0\}}$ is an action of $(\Z[i] \setminus \{0\}, \cdot)$ by continuous maps on $Y$. Thus $S_1 = \operatorname{id}_Y$ and $S_{n m} = S_n \circ S_m$ for all $m, n \in \Z[i] \setminus \{0\}$.

\item Let $(Y, S)$ be a multiplicative topological dynamical system. The action $S$ on $Y$ is called finitely generated if the set $\{S_p : p \in \mathcal{P}_{\Z[i]}\}$ is finite. Its distinct elements, denoted by $R_1, \ldots, R_d$, are called the generators of $S$.

\item Let $(Y, S)$ be a multiplicative topological dynamical system, and let $\nu$ be a Borel probability measure on $Y$. We say that $\nu$ pretends to be invariant under $S$ if there exists a set $P \subset \mathcal{P}_{\Z[i]}$ such that $\sum_{p \in \mathcal{P}_{\Z[i]} \setminus P} \mathcal{N}(p)^{-1} < \infty$ and $\nu$ is invariant under $S_p$ for every $p \in P$.

\item We call $(Y, S)$ strongly uniquely ergodic if there exists exactly one Borel probability measure on $Y$ that pretends to be invariant under $S$.
\end{enumerate}
\end{definition}

Below we quantify the robustness of multiplicative averages under the increasingly strong directional localization imposed by shrinking sectors. The comparison theorem itself requires only finite generation, with no ergodicity assumption. More precisely, we establish a quantitative total variation comparison between the sector and disk distributions of the prime-factor data determining $S_n$. This yields the following sector-disk estimate, uniform over systems with a bounded number of prime transformations and requiring neither an angular distribution hypothesis on the sets of primes inducing the individual transformations nor convergence of the disk averages.
We use the convention $\operatorname{sinc}(t) := \sin(t)/t$..

\begin{mainthm} \label{thm:sector-disk}
For every integer $d \geq 1$, set $A_d := \left(1 - \operatorname{sinc}\left(\pi/4d\right)\right) /128$ and $b_d := A_d/(2d + 2)$.
There exist constants $C_d > 0$ and $N_d \geq 3$, depending only on $d$, with the following property. Let $(Y, S)$ be a finitely generated multiplicative topological dynamical system satisfying
\begin{align}
\#\{S_p : p \in \mathcal{P}_{\Z[i]}\}
& \leq d.
\end{align}
Then, for every $g \in C(Y)$ and every $N \geq N_d$,
\begin{align}
\sup_{\substack{(\log N)^{-A_d} \leq \gamma \leq 2\pi \\ s \in [0, 2\pi) \\ y \in Y}} \left|\E_{n \in K_{s,\gamma}(N)} g(S_n y) - \E_{n \in B_N} g(S_n y)\right|
& \leq C_d \|g\|_\infty (\log N)^{-b_d}.
\end{align}
\end{mainthm}

Theorem \ref{thm:sector-disk} is independent of any convergence statement for the disk averages: it asserts that, in the stated shrinking range, angular localization does not change the asymptotic distribution of the finite prime factor data determining the action. Combining this comparison with the full-disk convergence theorem of Donoso, Le, Moreira, and Sun \cite{DLMS2024}*{Theorem E} gives the following ergodic consequence.

\begin{mainthm}\label{thm:multisec}
Let $d \geq 1$ be an integer, and let $A_d$ be as in Theorem \ref{thm:sector-disk}. Let $(Y, S)$ be a finitely generated and strongly uniquely ergodic multiplicative topological dynamical system, and let $\nu$ be the unique Borel probability measure on $Y$ that pretends to be invariant under $S$. Suppose 
\begin{align}
\#\{S_p : p \in \mathcal{P}_{\Z[i]}\}
& \leq d.
\end{align}
Then, for every $g \in C(Y)$ and every $y \in Y$,
\begin{align}
\sup_{\substack{(\log N)^{-A_d} \leq \gamma \leq 2\pi \\ s \in [0, 2\pi)}} \left|\E_{n \in K_{s,\gamma}(N)} g(S_n y) - \int_Y g \, d\nu\right|
& \longrightarrow 0,
\qquad 
N \rightarrow \infty.
\end{align}
\end{mainthm}

Theorem \ref{thm:sector-disk} gives a quantitative comparison that is uniform in the initial point and requires no ergodicity assumption, while Theorem \ref{thm:multisec} identifies the limiting value under strong unique ergodicity, at each fixed initial point and without a rate. The explicit exponents $A_d$ and $b_d$ are not optimized.

\medskip 

Bergelson and Richter \cite{BergelsonRichter} proved convergence of ergodic averages along the prime-factor counting function with multiplicity on the positive integers in every uniquely ergodic system. Their theorem gives a dynamical generalization of the prime number theorem and recovers the classical equidistribution results of Pillai--Selberg \cites{Pil40,Sel39} modulo integers and Erd\H{o}s--Delange \cites{Erd46,Del58} for irrational multiples.

Donoso, Le, Moreira, and Sun \cite{DLMS2024}*{Theorem E} established a Gaussian analogue for finitely generated, strongly uniquely ergodic multiplicative actions along dilated F{\o}lner sequences, including disks and fixed sectors. Taking $S_n := T^{\Omega(n)}$ gives convergence for fixed-sector $\Omega$-averages. In a different direction, Giannitsi, Miheisi, and Mousavi \cite{gaussiandivisor2024} proved a divisor-weighted pointwise ergodic theorem for additive $\Z[i]$-actions.

This work was obtained independently of C\'{e}spedes and Donoso \cite{CD26}, who proved convergence in uniquely ergodic systems for prime-factor counts, with and without multiplicity, averaged over norm-truncated ideals in arbitrary number fields. Over $\Z[i]$, their ideal $\Omega$-averages agree with full-disk element averages, since every nonzero ideal has exactly four generators and $\Omega$ is constant on associates. Both additive arguments establish shift invariance, but their arithmetic input consists of Hardy-Ramanujan and Sath\'e-Selberg type estimates, whereas ours uses prime-semiprime matching adapted to shrinking sectors. Their generality in the number field complements our control of angular restrictions. They also establish strong sweeping-out phenomena in invertible non-atomic ergodic systems.

The angular distribution of Gaussian primes in fixed sectors goes back to Hecke \cites{hecke-I,hecke-II}. Subsequent work treated increasingly narrow angular ranges. Harman and Lewis \cite{HL01} studied Gaussian primes in narrow sectors, while Huang, Liu, and Rudnick \cite{HLR20} and J{\"a}rviniemi and Ter{\"a}v{\"a}inen \cite{JT24} obtained almost-all results for Gaussian primes and Gaussian almost primes in narrow sectors, respectively.

Stucky's estimates \cite{stucky2021}, which impose simultaneous angular and radial restrictions, supply the prime counts used in our matching construction and prime distance bound. For the multiplicative comparison, we combine the latter bound with the quantitative Hal\'asz theorem for Gaussian ideals of Ku\'s \cite{Kus2026}*{Theorem 1.1}. Ku\'s also proves sectorial Hal\'asz estimates under angular non-pretentiousness hypotheses. Our use of his unrestricted Gaussian-ideal estimate is different: the finite-phase argument and prime-distance bound provide the angular cancellation needed here. The resulting angular cancellation estimate, Proposition \ref{prop:angularcancellation}, is uniform over unimodular completely multiplicative functions taking a bounded number of values on primes, without assumptions on the angular distribution of the sets of primes on which each value is attained.

\medskip

\noindent\textbf{Proof Outline}
\\
The proof for $\Omega$-averages builds on the prime-semiprime matching argument of Bergelson and Richter \cite{BergelsonRichter}. The guiding observation is that division by a prime and by a semiprime with nearly the same norm and argument produces almost the same averaging region, while the corresponding changes in $\Omega$ differ by one.  The guiding observation is that division by a prime and by a semiprime with nearly the same norm and argument produces almost the same averaging region, while the corresponding changes in $\Omega$ differ by one. The sectorial Tur\'an-Kubilius inequality in Proposition \ref{prop:TuranKubelius} allows us to exploit this observation by replacing sector averages with weighted averages conditioned on divisibility, with an error controlled by averaged gcd correlations.

For shrinking sectors, the matching must occur at an angular scale finer than the smallest allowed width, while the divisor norms must remain small enough to control boundary errors. Lemma \ref{lem:prime-semiprime} meets these requirements using Stucky's prime counts in fine annular sectors. The key construction is a local thinning of the larger-prime set that retains enough reciprocal-norm mass to keep the gcd correlations small, while making the resulting semiprimes sparse enough in each fine angular--radial cell to be matched there with distinct primes. This construction gives shift invariance uniformly against all test sequences bounded by $1$, yielding Theorem \ref{thm:omegaavg}. Consequently, every weak-$*$ limit of the dynamical empirical measures is $T$-invariant. Compactness and unique ergodicity then give Theorem \ref{thm:additivesec}, including uniformity in the initial point.

The multiplicative argument rests on a different observation: a fixed number of phases cannot approximate a phase traversing the circle with arbitrarily small average error. We minimize over the possible prime values before applying Stucky's counts, obtaining a prime-distance bound without any assumption on the angular distribution of the sets of primes on which each value is attained. Combined with the Gaussian Hal\'asz estimate of \cite{Kus2026}, this proves the angular cancellation in Proposition \ref{prop:angularcancellation}. Fej\'er approximation then converts cancellation of the nonzero angular frequencies into a scalar sector-disk comparison. Normalizing by the sector cardinality introduces a loss proportional to the reciprocal width, which the logarithmic shrinking range allows us to absorb.

To pass from scalar functions to the action, we encode $S_n$ by the residual unit and the prime-factor counts associated with the finitely many generators. Fourier inversion converts the scalar character estimates into a comparison of the distributions of these data, and a first-moment truncation upgrades the pointwise comparison to total variation. Since these data determine $S_n$, this yields Theorem \ref{thm:sector-disk}; the full-disk convergence theorem of \cite{DLMS2024} then gives Theorem \ref{thm:multisec}.
\smallskip

The results also describe the interaction between angular localization and the dynamics: in the subpolynomial shrinking range, relative angular position within the sector and the dynamical variable $T^{\Omega(n)} x$ are asymptotically independent (Corollary \ref{implication2}). In the same range, $\Omega(n)$ is equidistributed modulo every fixed integer and the Gaussian Liouville function exhibits cancellation (Corollary \ref{implication1}). In the logarithmic shrinking range, we also obtain joint equidistribution modulo every fixed integer of the prime-factor counts associated with any partition of the Gaussian primes into two associate-invariant classes with divergent reciprocal-norm sums, without any angular-distribution hypothesis on either class (Corollary \ref{implication4}).

\medskip We conclude our introductory section by collecting our definition and notational conventions below.

We write $\mathbb{N} := \{1, 2, \ldots\}$ and $\mathbb{N}_0 := \mathbb{N} \cup \{0\}$. All logarithms are natural. For a positive sequence $(\gamma_N)_{N \geq 3}$, the condition $\gamma_N^{-1} = N^{o(1)}$ means $\log(1/\gamma_N) = o(\log N)$.

The Gaussian norm is multiplicative, and $\Omega$ is completely additive: $\mathcal{N}(m n) = \mathcal{N}(m) \mathcal{N}(n)$ and $\Omega(m n) = \Omega(m) + \Omega(n)$ for nonzero $m, n \in \Z[i]$. Two nonzero Gaussian integers are associates if they differ by multiplication by a unit in $\{\pm 1, \pm i\}$. The function $\Omega : \Z[i] \setminus \{0\} \rightarrow \mathbb{N}_0$ vanishes on units. We write $\mathcal{P}_{\Z[i],2} := \Omega^{-1}(\{2\})$ for the Gaussian semiprimes.

For a finite nonempty set $E$ and $f : E \rightarrow \mathbb{C}$, define
\begin{align}
\E_{n \in E} f(n)
& := \frac{1}{|E|} \sum_{n \in E} f(n), \\ \E_{n \in E}^{\log} f(n)
& := \frac{\sum_{n \in E} f(n) \mathcal{N}(n)^{-1}}{\sum_{n \in E} \mathcal{N}(n)^{-1}},
\end{align}
where the logarithmic average requires $E \subset \Z[i] \setminus \{0\}$. We use $\one_{\mathcal{C}}$ and $\one_E$ for indicators of conditions and sets, respectively, and write $\nu(m) := \nu(\{m\})$ for the mass at $m$ of a finitely supported measure.

Norms of greatest common divisors and least common multiples are independent of the choice of associates. For nonzero $m, m' \in \Z[i]$, define $\Phi(m, m') := \mathcal{N}(\gcd(m, m')) - 1.$
For nonzero $n \in \Z[i]$, the Gaussian von Mangoldt function is $\Lambda(n) := \log \mathcal{N}(p)$ if $n$ is an associate of $p^k$ with $p \in \mathcal{P}_{\Z[i]}$ and $k \in \mathbb{N}$, and $\Lambda(n) := 0$ otherwise.

We use $\mathbb{T} := \R/\Z$, $\mathbb{T}_{\mathrm{ang}} := \R/(2\pi\Z)$, and $\mathbb{S}^1 := \{z \in \mathbb{C} : |z| = 1\}$. The normalized Haar measures on the first two circles are $m_{\mathbb{T}}$ and $m_{\mathrm{ang}}$, respectively. Arguments take values in $\mathbb{T}_{\mathrm{ang}}$, and angular intervals denote their images under the quotient map, with lengths measured in radians.

For $s \in \R$, $\gamma \in (0, 2\pi]$, and real $N \geq 1$, define
\begin{align}
K_{s,\gamma}
& := \{n \in \Z[i] \setminus \{0\} : \arg(n) \in [s, s + \gamma)\}, \\ K_{s,\gamma}(N)
& := \{n \in K_{s,\gamma} : \mathcal{N}(n)
\leq N\}.
\end{align}
We call $K_{s,\gamma}$ a sector of width $\operatorname{width}(K_{s,\gamma}) := \gamma$. The initial angle is taken modulo $2\pi$, and $\gamma = 2\pi$ gives the full sector $\Z[i] \setminus \{0\}$. We write $B_N := K_{0,2\pi}(N)$ for the full disk. Normalized averages are used only when the averaging sets are nonempty; by Lemma \ref{lem:countsec}, this holds uniformly for sufficiently large $N$ in the sector ranges of the main theorems.

For a sector $J$, a nonzero $m \in \Z[i]$, and real numbers $0 \leq \alpha < \beta$, set
\begin{align}
J_{[m]}
& := \{n \in \Z[i] \setminus \{0\} : m n \in J\}, \\ J[\alpha, \beta]
& := \{n \in J : \alpha
\leq \mathcal{N}(n)
< \beta\}.
\end{align}
The sector $J_{[m]}$ has the same width as $J$, and the half-open annular bounds in $J[\alpha, \beta]$ refer to the norm.

For $n \in K_{s,\gamma}$, its relative angular position $u_{s,\gamma}(n)$ is the unique $u \in [0, 1)$ satisfying $\arg(n) = s + \gamma u$ in $\mathbb{T}_{\mathrm{ang}}$. Admissible sector ranges are specified in the relevant statement or proof.

\section{Proof for the \texorpdfstring{$\Omega$}{Omega} averages}
\subsection{Auxiliary results}

The following proposition gives a sectorial Tur\'an-Kubilius inequality for finitely supported probability measures on the nonzero Gaussian integers. It compares a sector average with a weighted average of conditional averages over divisibility classes, with the error controlled by the corresponding gcd correlations.

\begin{proposition}\label{prop:TuranKubelius}
Let $J \subset \Z[i]$ be a sector, write $\gamma := \operatorname{width}(J)$, and let $\nu$ be a finitely supported probability measure on $\Z[i] \setminus \{0\}$. Let $R := \max\{\mathcal{N}(m) : m \in \operatorname{supp}(\nu)\}$. All sums over $m$ and $m'$ in this proposition and its proof are taken over $\operatorname{supp}(\nu)$. There exist absolute constants $c, C > 0$ so that, if $\gamma (N/R)^{1/2} \geq c$, the conditional averages below are well-defined, and for every arithmetic function $\alpha : \Z[i] \setminus \{0\} \rightarrow \mathbb{C}$ with $|\alpha| \leq 1$,
\begin{align}
\left| \E_{\substack{n \in J \\ \mathcal{N}(n) \leq N}} \hspace*{-0.3cm} \alpha(n) - \sum_m \nu(m) \hspace*{-0.3cm} \E_{\substack{n \in J \\ \mathcal{N}(n) \leq N \\ m \mid n}} \hspace*{-0.3cm} \alpha(n) \right|
& \leq \left( \sum_{m, m'} \nu(m) \nu(m') \Phi(m, m') + C \frac{R^{3/2}}{\gamma N^{1/2}} \right)^{1/2} \hspace*{-0.5cm} + C \frac{R^{1/2}}{\gamma N^{1/2}}.
\end{align}
\end{proposition}

A crucial step we will use is the following theorem by Stucky, which gives us the desired counts of primes in small sectorial cells.

\begin{theorem}\label{thm:stucky}
Fix $\zeta > 0$. Let $x > 1$, let $y = x^\theta$ with $7/10 < \theta < 1$, and let $I$ be the image of $[\varphi, \varphi + \delta)$ in $\mathbb{T}_{\mathrm{ang}}$, where $\varphi \in \R$ and $0 < \delta \leq \pi/2$. Suppose that $\delta y \geq x^{7/10 + \zeta}.$
Then, as $x \rightarrow \infty$,
\begin{align}
\psi(x, y; I)
& := \sum_{\substack{a \in \Z[i] \setminus \{0\} \\ x - y < \mathcal{N}(a) \leq x \\ \arg(a) \in I}} \Lambda(a) 
= \big(1 + o(1)\big) \frac{2\delta y}{\pi},
\end{align}
where, for fixed $\zeta$, the relative error tends to zero uniformly over $\varphi$, $\delta$, and $\theta$ satisfying the stated conditions.
\end{theorem}

This is the Gaussian-element formulation of \cite{stucky2021}*{Theorem 1.1}. In the ideal formulation, arguments are taken modulo $\pi/2$. Since $\delta \leq \pi/2$ and the angular interval is half-open, its projection onto $\R/((\pi/2)\Z)$ is injective. Every nonzero ideal whose argument lies in the projected interval has exactly one generator with argument in $I$. Thus the ideal sum transfers to the element sum without an additional factor of four, including when $I$ crosses a quadrant boundary. Changing Stucky's angular endpoint convention to $[\varphi, \varphi + \delta)$ affects only the two boundary rays. Inside the disk of radius $x^{1/2}$, these contain $O(x^{1/2})$ Gaussian integers, so their total weighted contribution is $O(x^{1/2}\log x)$, uniformly in the initial angle. Moreover, prime powers of exponent at least two contribute at most $O(x^{1/2}(\log x)^2)$, whether weighted by $\log \mathcal{N}(p)$ or by $\log \mathcal{N}(p^k)$, so the choice of weight on those terms does not affect the asymptotic. Both errors are negligible because
\begin{align}
\frac{x^{1/2}(\log x)^2}{\delta y}
& \leq x^{-1/5 - \zeta}(\log x)^2
\longrightarrow 0.
\end{align}

The main technical lemma is then as follows.

\begin{lemma}\label{lem:prime-semiprime}
Let $(\gamma_N)_{N \geq 3}$ be a sequence of positive real numbers satisfying $\gamma_N \rightarrow 0$ and $\gamma_N^{-1} = N^{o(1)}$ as $N \rightarrow \infty$, and fix $\varepsilon \in (0, 1)$. For $N \geq 3$, set
\begin{align}
\lambda_N := \frac{\gamma_N}{\exp(\sqrt{\log N})}, \qquad 
M_N := \lfloor \lambda_N^{-1}\rfloor, \qquad 
\rho_N := \exp(\lambda_N).
\end{align}
Then, for all sufficiently large $N$, there exist finite and nonempty sets $S_{1, N}, S_{2, N} \subset \Z[i] \setminus \{0\}$ with the following properties:
\begin{enumerate}[label=\textnormal{(\alph*)}]
\item \label{part:2.3a}$S_{1, N} \subset \mathcal{P}_{\Z[i]}$ and $S_{2, N} \subset \mathcal{P}_{\Z[i], 2}$.

\item \label{part:2.3b}There exists a decomposition of $\Z[i] \setminus \{0\}$ into $M_N$ sectors of equal angular width,
\begin{align}
\Z[i] \setminus \{0\}
& = \bigsqcup_{j = 1}^{M_N} Q_{j, N}, \quad \text{ where } \quad \operatorname{width}(Q_{j, N})
= \frac{2\pi}{M_N},
\end{align}
such that, for every $1 \leq j \leq M_N$ and every $k \in \mathbb{N}_0$,
\begin{align}
\big|S_{1, N} \cap Q_{j, N}[\rho_N^k, \rho_N^{k + 1}]\big|
& = \big|S_{2, N} \cap Q_{j, N}[\rho_N^k, \rho_N^{k + 1}]\big|.
\end{align}

\item \label{part:2.3c}For each $i = 1, 2$,
\begin{align}
\E_{m \in S_{i, N}}^{\log} \E_{m' \in S_{i, N}}^{\log} \Phi(m, m')
& \leq \varepsilon.
\end{align}

\item \label{part:2.3d}If $R_N := \max_{m \in S_{1, N} \cup S_{2, N}} \mathcal{N}(m)$, then the sets may be chosen so that as $N \rightarrow \infty$
\begin{align}
R_N
= O(N^{1/4 + o(1)}), \quad \frac{R_N^{3/2}}{\gamma_N N^{1/2}}
& \longrightarrow 0.
\end{align}
\end{enumerate}
\end{lemma}

\begin{lemma}\label{lem:matchingavg}
Let $N \geq 1$, $\gamma \in (0, 2\pi]$, and $R \geq 1$, and let $g : \Z[i] \setminus \{0\} \rightarrow \mathbb{C}$ satisfy $|g| \leq 1$. Suppose that $\gamma (N/R)^{1/2}$ is sufficiently large. For $s \in [0, 2\pi)$ and $m \in \Z[i] \setminus \{0\}$ with $\mathcal{N}(m) \leq R$, set
\begin{align}
F_{N, s, \gamma}(m)
& := \E_{\substack{n \in (K_{s, \gamma})_{[m]} \\ \mathcal{N}(n) \leq N/\mathcal{N}(m)}} g(n).
\end{align}
Let $m, m' \in \Z[i] \setminus \{0\}$ satisfy $\mathcal{N}(m), \mathcal{N}(m') \leq R$. Assume that their arguments differ modulo $2\pi$ by at most $h$, and that
\begin{align}
\left|\log \mathcal{N}(m) - \log \mathcal{N}(m')\right|
& \leq \eta
\leq \frac{1}{2}.
\end{align}
Then, uniformly in $s \in [0, 2\pi)$,
\begin{align}
\left|F_{N, s, \gamma}(m) - F_{N, s, \gamma}(m')\right|
& \lesssim \frac{h}{\gamma} + \eta + \frac{1}{\gamma}\left(\frac{N}{R}\right)^{-1/2},
\end{align}
where the implied constant is absolute.
\end{lemma}

\begin{lemma}\label{lem:uniformcomp}
Let $(\gamma_N)_{N \geq 3}$ be a sequence of positive real numbers satisfying $\gamma_N \rightarrow 0$ and $\gamma_N^{-1} = N^{o(1)}$ as $N \rightarrow \infty$. Fix $\varepsilon \in (0, 1)$, and let $S_{1, N}$ and $S_{2, N}$ be the finite nonempty sets supplied by Lemma \ref{lem:prime-semiprime} for this sequence and this choice of $\varepsilon$. Then, uniformly over all functions $g : \Z[i] \setminus \{0\} \rightarrow \mathbb{C}$ with $|g| \leq 1$,
\begin{align}
\sup_{\gamma_N \leq \gamma \leq 2\pi} \sup_{s \in [0, 2\pi)} \Bigg| \E_{m \in S_{1, N}}^{\log} \E_{\substack{n \in (K_{s, \gamma})_{[m]} \\ \mathcal{N}(n) \leq N/\mathcal{N}(m)}} g(n) - \E_{m' \in S_{2, N}}^{\log} \E_{\substack{n \in (K_{s, \gamma})_{[m']} \\ \mathcal{N}(n) \leq N/\mathcal{N}(m')}} g(n) \Bigg|
& \longrightarrow 0
\end{align}
as $N \rightarrow \infty$.
\end{lemma}

\subsection{Proofs of Auxiliary Results} 

Throughout this section, fix a sequence $(\gamma_N)_{N \geq 3}$ of positive real numbers satisfying the same assumptions as in Theorem \ref{thm:omegaavg}, namely $\gamma_N \rightarrow 0$ and $\gamma_N^{-1} = N^{o(1)}$ as $N \rightarrow \infty$.

\begin{lemma}\label{lem:countsec}
For every real $N > 0$, uniformly over all sectors $J \subset \Z[i]$, one has
\begin{align}
\sum_{\substack{n \in J \\ \mathcal{N}(n) \leq N}} 1
& = \frac{1}{2} \operatorname{width}(J) N + O(N^{1/2}),
\end{align}
where the implied constant is absolute.
\end{lemma}

\begin{proof}
See \cite{gaussiandivisor2024} \S2 for this estimate when $N \geq 1$. Although the uniformity is not stated explicitly there, the argument is uniform over the choice of sector: the discrepancy is controlled by a fixed-width neighborhood of the boundary of the truncated sector, whose total length is $O(N^{1/2})$ uniformly in both its position and angular width. For $0 < N < 1$, the counting sum is zero, and the asserted error bound is immediate.
\end{proof}

\begin{lemma}\label{lem:divisorest}
Let $m \in \Z[i]$ be nonzero, and let $J \subset \Z[i]$ be a sector. Write $\gamma := \operatorname{width}(J)$. Then, uniformly over $m$ and $J$, whenever $\gamma N^{1/2}$ is sufficiently large,
\begin{align}
\E_{\substack{n \in J \\ \mathcal{N}(n) \leq N}} \one_{m \mid n}
& = \frac{1}{\mathcal{N}(m)} + O\left(\frac{1}{\gamma \sqrt{N \mathcal{N}(m)}}\right),
\end{align}
where the implied constant is absolute.
\end{lemma}

\begin{proof}
The substitution $n = m \ell$ and Lemma \ref{lem:countsec}, using $\operatorname{width}(J_{[m]}) = \gamma$, give
\begin{align}
\E_{\substack{n \in J \\ \mathcal{N}(n) \leq N}} \one_{m \mid n}
& = \frac{\gamma N/(2 \N(m)) + O((N/\N(m))^{1/2})}{\gamma N/2 + O(N^{1/2})} 
= \frac{1}{\N(m)} + O\left(\frac{1}{\gamma \sqrt{N \N(m)}}\right),
\end{align}
where the denominator is comparable to $\gamma N$, and the last step uses $\N(m) \geq 1$.
\end{proof}

\begin{lemma}\label{lem:meansquare}
Let $\nu$ be a finitely supported probability measure on $\Z[i] \setminus \{0\}$, and set
\begin{align}
R
& := \max\{\mathcal{N}(m) : m \in \operatorname{supp}(\nu)\}.
\end{align}
Let $J \subset \Z[i]$ be a sector, and write $\gamma := \operatorname{width}(J)$. Then, uniformly over $\nu$ and $J$, whenever $\gamma N^{1/2}$ is sufficiently large,
\begin{align}
\E_{\substack{n \in J \\ \mathcal{N}(n) \leq N}} \left| \sum_m \nu(m) \big(1 - \mathcal{N}(m) \one_{m \mid n}\big) \right|^2
& = \sum_{m, m'} \nu(m) \nu(m') \Phi(m, m') + O\left(\frac{R^{3/2}}{\gamma N^{1/2}}\right).
\end{align}
\end{lemma}

\begin{proof}
Let $E := \{n \in J : \mathcal{N}(n) \leq N\}$. For $m, m' \in \operatorname{supp}(\nu)$, write $M := \mathcal{N}(m)$, $M' := \mathcal{N}(m')$, $G := \mathcal{N}(\gcd(m, m'))$, and $L := \lcm(m, m')$. Since $\one_{m \mid n} \one_{m' \mid n} = \one_{L \mid n}$ and $\mathcal{N}(L) = M M'/G$, expanding the product and applying Lemma \ref{lem:divisorest} to $m$, $m'$, and $L$ gives
\begin{align}
& \E_{n \in E} \big(1 - M \one_{m \mid n}\big) \big(1 - M' \one_{m' \mid n}\big) \\
& = 1 - M \E_{n \in E} \one_{m \mid n} - M' \E_{n \in E} \one_{m' \mid n} + M M' \E_{n \in E} \one_{L \mid n} \\
& = G - 1 + O\left(\frac{M^{1/2} + (M')^{1/2} + (M M' G)^{1/2}}{\gamma N^{1/2}}\right) \\
& = \Phi(m, m') + O\left(\frac{R^{3/2}}{\gamma N^{1/2}}\right),
\end{align}
where the last step uses $M, M', G \leq R$ and $R \geq 1$. The inner sum in the statement is real-valued, so expanding its square and summing the preceding identity with weights $\nu(m) \nu(m')$ proves the result, since $\sum_{m, m'} \nu(m) \nu(m') = 1$.
\end{proof}

\begin{proof}[Proof of Proposition \ref{prop:TuranKubelius}]
Let $E := \{n \in J : \mathcal{N}(n) \leq N\}$. For $m \in \operatorname{supp}(\nu)$, set $M_m := \mathcal{N}(m)$, $d_m := \E_{n \in E} \one_{m \mid n}$, and $b_m := \E_{n \in E} \alpha(n) \one_{m \mid n}$. Lemma \ref{lem:divisorest} gives
\begin{align}
|1 - M_m d_m|
& \lesssim \frac{M_m^{1/2}}{\gamma N^{1/2}}
\leq \frac{R^{1/2}}{\gamma N^{1/2}}.
\end{align}
The hypothesis $\gamma (N/R)^{1/2} \geq c$, with $c$ sufficiently large, therefore gives $d_m \geq (2 M_m)^{-1} > 0$. Thus all conditional averages in the statement are well-defined and equal to $b_m/d_m$.

Since $|\alpha| \leq 1$, Cauchy-Schwarz and Lemma \ref{lem:meansquare} give
\begin{align}
& \left|\E_{n \in E} \alpha(n) - \sum_m \nu(m) M_m b_m\right| \\
& = \left|\E_{n \in E} \alpha(n) \sum_m \nu(m) \big(1 - M_m \one_{m \mid n}\big)\right| \\
& \leq \left(\sum_{m, m'} \nu(m) \nu(m') \Phi(m, m') + C \frac{R^{3/2}}{\gamma N^{1/2}}\right)^{1/2}.
\end{align}
Moreover, $|b_m| \leq d_m$, so the normalization error satisfies
\begin{align}
\left|\frac{b_m}{d_m} - M_m b_m\right|
& = \left|\frac{b_m}{d_m}\right| |1 - M_m d_m| 
\leq |1 - M_m d_m|
\lesssim \frac{R^{1/2}}{\gamma N^{1/2}}.
\end{align}
Averaging the last estimate with respect to $\nu$ and combining it with the preceding Cauchy-Schwarz bound by the triangle inequality proves the proposition.
\end{proof}

We first isolate the geometric comparison underlying the proof.

\begin{proof}[Proof of Lemma \ref{lem:matchingavg}]
For $m, m' \in \Z[i] \setminus \{0\}$, set $X_m := N/\mathcal{N}(m),$ $X_{m'} := \frac{N}{\mathcal{N}(m')},$ and write $X_- := \min\{X_m, X_{m'}\}$ and $X_+ := \max\{X_m, X_{m'}\}$. The assumption on the norms gives $X_+ /X_- \leq e^\eta,$ and therefore, since $\eta \leq 1/2$, we have $X_+ - X_- \lesssim \eta X_-,$ and $X_+ \lesssim X_-.$ Moreover, since $\mathcal{N}(m), \mathcal{N}(m') \leq R$, $X_- \geq \frac{N}{R}.$

For $k \in \{m, m'\}$, define $A_k := \left\{n \in (K_{s, \gamma})_{[k]} : \mathcal{N}(n) \leq X_k\right\}$. The initial angles of these two sectors differ by at most $h$. Hence $A_m \mathbin{\triangle} A_{m'}$ is contained in the union of angular wedges of total width $O(h)$, truncated at norm $X_+$, together with an annular sector of width $\gamma$ whose difference in norm cutoffs is $O(\eta X_-)$. Lemma \ref{lem:countsec} and $X_+ \lesssim X_-$ give
\begin{align}
\big|A_m \mathbin{\triangle} A_{m'}\big|
& \lesssim \big(h + \eta \gamma\big)X_- + X_-^{1/2}.
\end{align}

On the other hand, Lemma \ref{lem:countsec} gives
\begin{align}
|A_m|
& = \frac{\gamma}{2}X_m + O(X_m^{1/2}), & |A_{m'}|
& = \frac{\gamma}{2}X_{m'} + O(X_{m'}^{1/2}).
\end{align}
Since $\gamma (N/R)^{1/2}$ is sufficiently large and $X_- \geq N/R$, it follows that $\min\big\{|A_m|, |A_{m'}|\big\} \gtrsim \gamma X_-.$ In particular, both averaging sets are nonempty.

Since $|g| \leq 1$,
\begin{align}
\left|F_{N, s, \gamma}(m) - F_{N, s, \gamma}(m')\right|
& \leq \frac{2\big|A_m \mathbin{\triangle} A_{m'}\big|}{\min\big\{|A_m|, |A_{m'}|\big\}} \\
& \lesssim \frac{h}{\gamma} + \eta + \frac{1}{\gamma X_-^{1/2}} \\
& \leq \frac{h}{\gamma} + \eta + \frac{1}{\gamma}\left(\frac{N}{R}\right)^{-1/2}.
\end{align}
This proves the lemma.
\end{proof}

We now apply Lemma \ref{lem:matchingavg} to the matched elements supplied by Lemma \ref{lem:prime-semiprime}.

\begin{proof}[Proof of Lemma \ref{lem:uniformcomp} assuming Lemma \ref{lem:prime-semiprime}]
Fix $\varepsilon \in (0, 1)$. For all sufficiently large $N$, let $S_{1,N}$ and $S_{2,N}$ be supplied by Lemma \ref{lem:prime-semiprime}. Set $\omega_N := 2\pi/M_N$ and $R_N := \max_{m \in S_{1,N} \cup S_{2,N}} \mathcal{N}(m)$. By condition \ref{part:2.3b}, matching the elements within each annular-sector cell gives a bijection $\tau_N : S_{1,N} \rightarrow S_{2,N}$ such that matched elements belong to the same cell. Their arguments therefore differ modulo $2\pi$ by at most $\omega_N$, and their logarithmic norms differ by at most $\log \rho_N = \lambda_N$.

Let $|g| \leq 1$, and use the notation $F_{N,s,\gamma}$ from Lemma \ref{lem:matchingavg}. Since $R_N \geq 1$, condition \ref{part:2.3d} implies $\gamma_N (N/R_N)^{1/2} \rightarrow \infty$. Thus, for all sufficiently large $N$, Lemma \ref{lem:matchingavg} gives
\begin{align}
& \sup_{\gamma_N \leq \gamma \leq 2\pi} \sup_{s \in [0, 2\pi)} \max_{m \in S_{1,N}} \left|F_{N,s,\gamma}(m) - F_{N,s,\gamma}(\tau_N(m))\right| \\
& \lesssim \frac{\omega_N}{\gamma_N} + \lambda_N + \frac{R_N^{1/2}}{\gamma_N N^{1/2}}
\longrightarrow 0.
\end{align}
Here $\omega_N/\gamma_N = O(\exp(-\sqrt{\log N}))$, $\lambda_N \rightarrow 0$, and the last term tends to zero by condition \ref{part:2.3d}. The estimate is uniform over all such $g$.

For $i \in \{1, 2\}$, define
\begin{align}
W_i
& := \sum_{a \in S_{i,N}} \mathcal{N}(a)^{-1}, \qquad w_i(m)
:= \frac{\mathcal{N}(m)^{-1}}{W_i}, \quad m \in S_{i,N}.
\end{align}
The radial matching gives $\rho_N^{-1} \leq \mathcal{N}(\tau_N(m))/\mathcal{N}(m) \leq \rho_N$ and, after summing the corresponding reciprocal-norm inequalities, $\rho_N^{-1} \leq W_2/W_1 \leq \rho_N$. Consequently,
\begin{align}
\rho_N^{-2}
& \leq \frac{w_1(m)}{w_2(\tau_N(m))}
= \frac{\mathcal{N}(\tau_N(m))}{\mathcal{N}(m)} \frac{W_2}{W_1}
\leq \rho_N^2,
\end{align}
and hence
\begin{align}
\sum_{m \in S_{1,N}} |w_1(m) - w_2(\tau_N(m))|
& \leq (\rho_N^2 - 1) \sum_{m \in S_{1,N}} w_2(\tau_N(m))
= \rho_N^2 - 1.
\end{align}

For fixed $s$ and $\gamma$, abbreviate $F := F_{N,s,\gamma}$ and set $B_i(N) := \sum_{m \in S_{i,N}} w_i(m) F(m)$. These are the two logarithmic averages in the statement. Since $|F| \leq 1$, the bijection $\tau_N$ and the triangle inequality give
\begin{align}
& |B_1(N) - B_2(N)| \\
& \leq \max_{m \in S_{1,N}} |F(m) - F(\tau_N(m))| + \sum_{m \in S_{1,N}} |w_1(m) - w_2(\tau_N(m))| \\
& \leq \max_{m \in S_{1,N}} |F(m) - F(\tau_N(m))| + \rho_N^2 - 1.
\end{align}
The first term tends to zero uniformly over $s$, $\gamma$, and $g$ by the matching estimate above, while $\rho_N \rightarrow 1$. This proves the lemma.
\end{proof}

\medskip

We now want to understand how to produce condition \ref{part:2.3c} in Lemma \ref{lem:prime-semiprime}.

\begin{lemma}\label{lem2.4}
Let $\varepsilon \in (0, 1)$, let $P \subset \mathcal{P}_{\Z[i]}$ be finite and nonempty, and suppose that
\begin{align}
\sum_{n \in P} \frac{1}{\mathcal{N}(n)}
& \geq \frac{4}{\varepsilon}.
\end{align}
Then
\begin{align}
\E_{n \in P}^{\log} \E_{n' \in P}^{\log} \Phi(n, n')
& \leq \varepsilon.
\end{align}
\end{lemma}

\begin{proof}
We begin by noting that for Gaussian primes $n, n'$, one has $\Phi(n, n') \neq 0$ if and only if $n$ and $n'$ are associates, and in this case $\Phi(n, n') = \mathcal{N}(n) - 1$. Write $W(P) := \sum_{n \in P}\mathcal{N}(n)^{-1}$. Since every Gaussian prime has at most four associates in $P$,
\begin{align}
\E_{n \in P}^{\log} \E_{n' \in P}^{\log} \Phi(n, n')
& \leq 4 W(P)^{-2} \sum_{n \in P} \frac{\mathcal{N}(n) - 1}{\mathcal{N}(n)^2}
\leq 4 W(P)^{-1}
\leq \varepsilon.
\end{align}
\end{proof}

\begin{lemma}\label{lem2.5}
Let $\varepsilon \in (0, 1)$, and let $P_1, P_2 \subset \mathcal{P}_{\Z[i]}$ be finite and nonempty. Suppose that $P_1 \cap P_2 = \varnothing$ and that no two distinct elements of $P_1 \cup P_2$ are associates. Set
\begin{align}
P'
& := P_1 \cdot P_2
:= \{p q : p \in P_1, \ q \in P_2\}.
\end{align}
If
\begin{align}
\sum_{p \in P_1} \frac{1}{\mathcal{N}(p)}
& \geq \frac{12}{\varepsilon} && \text{ and } & \sum_{q \in P_2} \frac{1}{\mathcal{N}(q)}
& \geq \frac{12}{\varepsilon},
\end{align}
then
\begin{align}
\E_{n \in P'}^{\log} \E_{n' \in P'}^{\log} \Phi(n, n')
& \leq \varepsilon.
\end{align}
\end{lemma}

\begin{proof}
Write
\begin{align}
A
& := \sum_{p \in P_1} \frac{1}{\mathcal{N}(p)}, & B
& := \sum_{q \in P_2} \frac{1}{\mathcal{N}(q)}.
\end{align}
The hypotheses on associates imply that the product map $P_1 \times P_2 \rightarrow P'$ is bijective. In particular,
\begin{align}
\sum_{n \in P'} \frac{1}{\mathcal{N}(n)}
& = A B.
\end{align}
For $p, p' \in P_1$ and $q, q' \in P_2$, the quantity $\Phi(p q, p' q')$ can be nonzero only if $p = p'$ or $q = q'$. The contributions from the cases $p = p'$ and $q \neq q'$, $q = q'$ and $p \neq p'$, and $p = p'$ and $q = q'$ are bounded by $A B^2$, $A^2 B$, and $A B$, respectively. Therefore,
\begin{align}
\E_{n \in P'}^{\log} \E_{n' \in P'}^{\log} \Phi(n, n')
& \leq (A B)^{-2} \big(A B^2 + A^2 B + A B\big)   < \varepsilon.
\end{align}
\end{proof}

To better understand condition \ref{part:2.3d} in Lemma \ref{lem:prime-semiprime}, we record a consequence of the prime number theorem in fixed sectors that allows us to obtain arbitrarily large logarithmic weight while keeping all primes below a prescribed small power of $N$.

\begin{lemma}\label{lem:largerecipricalmass}
Let $V \subset \Z[i] \setminus \{0\}$ be a fixed sector of positive width strictly smaller than $\pi/2$, whose angular closure is contained in the interior of one quadrant. With $\lambda_N$ as in Lemma \ref{lem:prime-semiprime}, define, for all sufficiently large $N$,
\begin{align}
X_N & := \exp\Big((\log N)^{1/2}(1 + \log(1/\lambda_N))^{1/2}\Big), \\
Y_N & := \exp\Big((\log N)^{3/4}(1 + \log(1/\lambda_N))^{1/4}\Big), \\
Z_N & := N^{1/4}.
\end{align}
Then, for all sufficiently large $N$, $X_N < Y_N < Z_N,$ and  as $N \rightarrow \infty$,
\begin{align}
\sum_{\substack{p \in \mathcal{P}_{\Z[i]} \cap V \\ X_N < \mathcal{N}(p) < Y_N}} 
\frac{1}{\mathcal{N}(p)}
& \longrightarrow \infty, &
\sum_{\substack{p \in \mathcal{P}_{\Z[i]} \cap V \\ Y_N < \mathcal{N}(p) < Z_N}} \frac{1}{\mathcal{N}(p)}
& \longrightarrow \infty.
\end{align}
\end{lemma}

\begin{proof}
We note that
\begin{align}
\frac{\log Y_N}{\log X_N}
= \Big(\frac{\log N}{1 + \log(1/\lambda_N)}\Big)^{1/4}
\rightarrow \infty,
\end{align}
since $\log(1/\lambda_N) = \sqrt{\log N} + \log(1/\gamma_N) = o(\log N)$ (by the hypothesis $\gamma_N^{-1} = N^{o(1)}$). Similarly,
\begin{align}
\frac{\log Z_N}{\log Y_N}
& = \frac{1}{4}\Big(\frac{\log N}{1 + \log(1/\lambda_N)}\Big)^{1/4}
\longrightarrow \infty.
\end{align}
Next, write $\pi_V(t) := \#\{p \in \mathcal{P}_{\Z[i]} \cap V : \mathcal{N}(p) \leq t\}$. By the prime number theorem for fixed sectors (see \cites{hecke-I,hecke-II}), there exist constants $c_V, C_V > 0$ such that, for all sufficiently large $t$,
\begin{align}
c_V \frac{t}{\log t}
& \leq \pi_V(t)
\leq C_V \frac{t}{\log t}.
\end{align}
For $y_2 > y_1$ with $y_1$ sufficiently large, partial summation therefore gives
\begin{align}
\sum_{\substack{p \in \mathcal{P}_{\Z[i]} \cap V \\ y_1 < \mathcal{N}(p) < y_2}} \frac{1}{\mathcal{N}(p)}
& = \frac{\pi_V(y_2^-)}{y_2} - \frac{\pi_V(y_1)}{y_1} + \int_{y_1}^{y_2} \frac{\pi_V(t)}{t^2} \, dt \\
& \geq c_V \int_{y_1}^{y_2} \frac{dt}{t \log t} - \frac{C_V}{\log y_1} \\
& = c_V \log\left(\frac{\log y_2}{\log y_1}\right) - \frac{C_V}{\log y_1},
\end{align}
where $\pi_V(y_2^-)$ counts the Gaussian primes in $V$ with norm strictly smaller than $y_2$. Applying this estimate with $(y_1, y_2) = (X_N, Y_N)$ and $(y_1, y_2) = (Y_N, Z_N)$ proves both required reciprocal-mass divergences, since $X_N, Y_N \rightarrow \infty$ and both logarithmic ratios tend to infinity by the preceding calculations.
\end{proof}

\begin{proof}[Proof of Lemma \ref{lem:prime-semiprime}]
Fix $\varepsilon \in (0, 1)$. For all sufficiently large $N$, partition $\Z[i] \setminus \{0\}$ into $M_N$ sectors $Q_{1, N}, \ldots, Q_{M_N, N}$ of equal angular width $\omega_N := 2\pi/M_N$. Since $\log \rho_N = \lambda_N$ and $M_N = \lfloor \lambda_N^{-1}\rfloor$, we have $\omega_N \simeq \log \rho_N \simeq \lambda_N$. Also, we have
$\rho_N - 1 = \exp(\lambda_N) - 1 \simeq \lambda_N.$

We first record the consequence of Theorem \ref{thm:stucky} that will be used below. Fix a sector $V$ of positive width strictly smaller than $\pi/2$ whose angular closure is contained in the interior of one quadrant, and set $X_N, Y_N, Z_N$ as in Lemma \ref{lem:largerecipricalmass}.

The relation $\lambda_N/\gamma_N = \exp(-\sqrt{\log N}) \rightarrow 0$ ensures that the angular matching occurs on a scale finer than the smallest admissible sector width. The divergent ratios $\log Y_N/\log X_N$ and $\log Z_N/\log Y_N$ provide large reciprocal-norm sums in both prime ranges, as established in Lemma \ref{lem:largerecipricalmass}. Finally, the bound $\rho_N Y_N Z_N = N^{1/4 + o(1)}$ will control the norms of both the semiprimes and their matched primes, ensuring that the error in condition \ref{part:2.3d} tends to zero.

We claim that there exist constants $c, C > 0$ such that, for all sufficiently large $N$, uniformly over $1 \leq j \leq M_N$ and all $x \geq 1$ satisfying $Y_N/\rho_N \leq x \leq Y_N Z_N,$
one has
\begin{align}
c \frac{\omega_N(\rho_N - 1)x}{\log x}
\leq |\mathcal{P}_{\Z[i]} \cap Q_{j, N}[x, \rho_N x]|
\leq C \frac{\omega_N(\rho_N - 1)x}{\log x}.
\label{eq:stucky-cell-count}
\end{align}
Indeed, for $x$ in this range, one has that
\begin{align}
\frac{\log(1/\lambda_N)}{\log x}
\leq \frac{\log(1/\lambda_N)}{\log Y_N - \log \rho_N}
\rightarrow 0
\end{align}
and so uniformly in this range $\lambda_N = x^{-o(1)}$. Set $x' := \rho_N x$ and $y := (\rho_N - 1)x$. The length of the norm interval defining $Q_{j, N}[x, \rho_N x]$ is $(\rho_N - 1)x \simeq \lambda_N x = x^{1 - o(1)}$, while its angular width is $\omega_N \simeq \lambda_N$. Consequently,
\begin{align}
\omega_N y
\simeq x^{1 - o(1)}.
\end{align}
We apply Theorem \ref{thm:stucky} with parameters $(\zeta, x', y, \theta)$, taking $\zeta = 1/10$. Set
\begin{align}
\theta
& := \frac{\log y}{\log x'}
= \frac{\log((\rho_N - 1)x)}{\log(\rho_N x)}
= 1 - o(1),
\end{align}
uniformly over the stated range of $x$. Since $0 < y < x'$ and $\theta \rightarrow 1$ uniformly, we have $7/10 < \theta < 1$ for all sufficiently large $N$. Moreover, the definition gives $y = (x')^\theta$, as required. We may then apply the theorem: since
$\omega_N y \simeq \lambda_N x^{1 - o(1)} \simeq x^{1 - o(1)} \geq (x')^{8/10}$
for sufficiently large $N$, one has that
\begin{align}
\sum_{\substack{x < \mathcal{N}(a) \leq \rho_N x \\ a \in Q_{j, N}}} \Lambda(a)
= (1 + o(1))\frac{2\omega_N (\rho_N - 1)x}{\pi},
\end{align}
uniformly in all parameters. We next pass to the half-open radial convention used in the definition of $Q_{j, N}[x, \rho_N x]$. Replacing $x < \mathcal{N}(a) \leq \rho_N x$ by $x \leq \mathcal{N}(a) < \rho_N x$ changes the sum only on the two boundary circles. Since $\rho_N \rightarrow 1$, these circles contain $O(x^{1/2})$ Gaussian integers in total, and their weighted contribution is $O(x^{1/2}\log x)$, uniformly over the same range of $x$. The contribution of Gaussian prime powers of exponent at least two is $O(x^{1/2}(\log x)^2)$. Both contributions are negligible relative to the main term, since
\begin{align}
\frac{x^{1/2}(\log x)^2}{\omega_N(\rho_N - 1)x}
& = x^{-1/2 + o(1)}
\longrightarrow 0
\end{align}
uniformly in the stated range. Consequently,
\begin{align}
\sum_{p \in \mathcal{P}_{\Z[i]} \cap Q_{j, N}[x, \rho_N x]} \log \mathcal{N}(p)
& = \big(1 + o(1)\big) \frac{2\omega_N(\rho_N - 1)x}{\pi},
\end{align}
uniformly over $j$ and $x$. Every prime in this sum satisfies $\log \mathcal{N}(p) = \log x + O(\lambda_N) = (1 + o(1))\log x$, uniformly over the same parameters. Dividing by $\log x$ therefore gives \eqref{eq:stucky-cell-count}.

Let $L := 12/\varepsilon$. By Lemma \ref{lem:largerecipricalmass}, applied to the fixed sector $V$, we obtain that
\begin{align}
\sum_{\substack{p \in \mathcal{P}_{\Z[i]} \cap V \\ X_N < \mathcal{N}(p) \leq Y_N}} \frac{1}{\mathcal{N}(p)}
\longrightarrow \infty
\qquad \text{ and } 
\sum_{\substack{q \in \mathcal{P}_{\Z[i]} \cap V \\ Y_N < \mathcal{N}(q) \leq Z_N}} \frac{1}{\mathcal{N}(q)}
\longrightarrow \infty.
\end{align}
Consequently, for all sufficiently large $N$, we may choose a finite set
\begin{align}
P_{1, N}
& \subset \{p \in \mathcal{P}_{\Z[i]} \cap V : X_N
< \mathcal{N}(p)
\leq Y_N\}
\end{align}
such that
\begin{align}
L & \leq \sum_{p \in P_{1, N}} \frac{1}{\mathcal{N}(p)}
\leq L + 1.
\label{eq:p1mass}
\end{align}

We next choose $P_{2, N}$ in a locally sparse manner. Set $H_N := \{q \in \mathcal{P}_{\Z[i]} \cap V : Y_N < \mathcal{N}(q) \leq Z_N\},$ and for $1 \leq j \leq M_N$ and $r \in \mathbb{N}_0$ write $H_{j, r} := H_N \cap Q_{j, N}[\rho_N^r, \rho_N^{r + 1}].$ Let $\eta \in (0, 1)$ be fixed independently of $N$, to be chosen sufficiently small in terms of $\varepsilon$ below. For each $j, r$, choose $P_{2, j, r} \subset H_{j, r}$ with $|P_{2, j, r}| = \left\lfloor \eta |H_{j, r}| \right\rfloor,$ and define
\begin{align}
P_{2, N}
& := \bigsqcup_{j, r} P_{2, j, r}.
\end{align}

We claim first that
\begin{align}
\sum_{q \in P_{2, N}} \frac{1}{\mathcal{N}(q)}
& \longrightarrow \infty.
\label{eq:P2-mass-diverges}
\end{align}
Let $\mathcal{J}_N := \{(j, r) : 1 \leq j \leq M_N, \ r \in \mathbb{N}_0, \ H_{j, r} \neq \varnothing\}$ be the set of indices of occupied annular-sector cells, and set $J_N := |\mathcal{J}_N|$. For every $(j, r) \in \mathcal{J}_N$,
\begin{align}
\sum_{q \in P_{2, j, r}} \frac{1}{\mathcal{N}(q)}
& \geq \frac{\lfloor \eta |H_{j, r}| \rfloor}{\rho_N^{r + 1}}
\geq \frac{\eta}{\rho_N} \sum_{q \in H_{j, r}} \frac{1}{\mathcal{N}(q)} - \frac{1}{\rho_N^{r + 1}}.
\label{eq:cellwiseLowerBound}
\end{align}

Every occupied cell satisfies $\rho_N^{r + 1} > Y_N$ and $\rho_N^r \leq Z_N$. Thus the number of possible radial indices is at most $2 + \log(Z_N/Y_N)/\log \rho_N$, and consequently
\begin{align}
J_N
& \leq M_N\left(2 + \frac{\log(Z_N/Y_N)}{\log \rho_N}\right)
= O(\lambda_N^{-2}\log N).
\end{align}
Since $H_N$ is nonempty for all sufficiently large $N$, we have $J_N \geq 1$ in this range. The definition of $Y_N$ and the relation $\log(1/\lambda_N) = o(\log N)$ give
\begin{align}
\log J_N
& = O\big(\log(1/\lambda_N) + \log \log N\big)
= o(\log Y_N).
\end{align}
Since $Y_N \rightarrow \infty$, it follows that $J_N/Y_N \rightarrow 0$.

Summing \eqref{eq:cellwiseLowerBound} only over occupied cells, and using $\rho_N^{r + 1} > Y_N$ for every $(j, r) \in \mathcal{J}_N$, we obtain
\begin{align}
\sum_{q \in P_{2, N}} \frac{1}{\mathcal{N}(q)}
& = \sum_{(j, r) \in \mathcal{J}_N} \sum_{q \in P_{2, j, r}} \frac{1}{\mathcal{N}(q)} \\
& \geq \frac{\eta}{\rho_N} \sum_{q \in H_N} \frac{1}{\mathcal{N}(q)} - \sum_{(j, r) \in \mathcal{J}_N} \frac{1}{\rho_N^{r + 1}} \\
& \geq \frac{\eta}{\rho_N} \sum_{q \in H_N} \frac{1}{\mathcal{N}(q)} - \frac{J_N}{Y_N}.
\end{align}
Since $\eta > 0$ is fixed, $\rho_N \rightarrow 1$, $J_N/Y_N \rightarrow 0$, and the reciprocal-norm sum over $H_N$ tends to infinity, this proves \eqref{eq:P2-mass-diverges}. Thus, after increasing $N$ if necessary,
\begin{align}
\sum_{q \in P_{2, N}} \frac{1}{\mathcal{N}(q)}
& \geq L.
\label{eq:P2-mass}
\end{align}

The construction also gives the required local sparsity. Namely, if $J$ is any sector of width $\omega_N$ and $a > 0$, then whenever $P_{2, N} \cap J[a, \rho_N a] \neq \varnothing,$ the set $J[a, \rho_N a]$ meets at most two of the sectors $Q_{j, N}$ and at most two of the radial intervals $[\rho_N^r, \rho_N^{r + 1})$. Only cells with $H_{j, r} \neq \varnothing$ contribute. For each such cell meeting $J[a, \rho_N a]$, its lower norm endpoint $t := \rho_N^r$ satisfies $Y_N/\rho_N < t \leq Z_N$ and $a/\rho_N < t < \rho_N a$. Thus \eqref{eq:stucky-cell-count} applies to every contributing cell, and its upper bound is at most a constant multiple of $\omega_N(\rho_N - 1)a/\log a$, uniformly over these cells. Therefore,
\begin{align}
\big|P_{2, N} \cap J[a, \rho_N a]\big|
& \leq C \frac{\eta \omega_N(\rho_N - 1) a}{\log a}.
\label{eq:P2-local-sparsity}
\end{align}

The sets $P_{1, N}$ and $P_{2, N}$ are disjoint because their norm ranges are disjoint. Moreover, since both are contained in $V$ and $\operatorname{width}(V) < \pi/2$, no two distinct elements of $P_{1, N} \cup P_{2, N}$ are associates. Define $S_{2, N} := P_{1, N} \cdot P_{2, N}$. The product map $P_{1, N} \times P_{2, N} \rightarrow S_{2, N}$ is bijective, and $S_{2, N} \subset \mathcal{P}_{\Z[i], 2}$. By \eqref{eq:p1mass}, \eqref{eq:P2-mass}, and Lemma \ref{lem2.5},
\begin{align}
\E_{m \in S_{2, N}}^{\log} \E_{m' \in S_{2, N}}^{\log} \Phi(m, m')
& \leq \varepsilon.
\label{eq:S2-Phi-bound}
\end{align}

We now construct $S_{1, N}$. For $1 \leq j \leq M_N$ and $r \in \mathbb{N}_0$, set $T_r^j := S_{2, N} \cap Q_{j, N}[\rho_N^r, \rho_N^{r + 1}].$ Suppose that $T_r^j$ is nonempty, and write $x := \rho_N^r$. For every $p \in P_{1, N}$, the elements $q \in P_{2, N}$ for which $p q \in T_r^j$ satisfy
\begin{align}
q
& \in (Q_{j, N})_{[p]} \left[ \frac{x}{\mathcal{N}(p)}, \frac{\rho_N x}{\mathcal{N}(p)} \right].
\end{align}
Let
\begin{align}
P_{1, N}(j, r)
& := \left\{ p \in P_{1, N} : P_{2, N} \cap (Q_{j, N})_{[p]} \left[ \frac{x}{\mathcal{N}(p)}, \frac{\rho_N x}{\mathcal{N}(p)} \right]
\neq \varnothing \right\}.
\end{align}
Since $(Q_{j, N})_{[p]}$ has width $\omega_N$, \eqref{eq:P2-local-sparsity} gives
\begin{align}
|T_r^j|
& \leq C \eta \omega_N(\rho_N - 1)x \sum_{p \in P_{1, N}(j, r)} \frac{1}{ \mathcal{N}(p)\log(x/\mathcal{N}(p)) }.
\end{align}
For every $p \in P_{1, N}(j, r)$, there exists $q \in P_{2, N}$ with $p q \in T_r^j$. Since $\mathcal{N}(p) \leq Y_N < \mathcal{N}(q)$ and $x \leq \mathcal{N}(p q) < \rho_N x$, we have $\mathcal{N}(p)^2 < \mathcal{N}(p q) < \rho_N x,$ and so, for all sufficiently large $N$,
\begin{align}
\log\left(\frac{x}{\mathcal{N}(p)}\right)
& > \frac{1}{2} \log\left(\frac{x}{\rho_N}\right)
\geq \frac{1}{3} \log x.
\end{align}
Consequently, by \eqref{eq:p1mass},
\begin{align}
|T_r^j|
& \leq 3 C \eta \frac{\omega_N(\rho_N - 1) x}{\log x} \sum_{p \in P_{1, N}(j, r)} \frac{1}{\mathcal{N}(p)} \\
& \leq 3 C \eta (L + 1) \frac{\omega_N(\rho_N - 1)x}{\log x}.
\end{align}

Choose $\eta := \min\{1/2, c/(6 C(L + 1))\}$ so that $|T_r^j| \leq \frac{c}{2}\frac{\omega_N(\rho_N - 1)x}{\log x}$. Since $T_r^j$ is nonempty, we also have $X_N Y_N/\rho_N < x \leq Y_N Z_N.$
Thus \eqref{eq:stucky-cell-count} applies and gives
\begin{align}
\big|\mathcal{P}_{\Z[i]} \cap Q_{j, N}[\rho_N^r, \rho_N^{r + 1}]\big|
& \geq c \frac{\omega_N(\rho_N - 1) x}{\log x}
\geq |T_r^j|.
\end{align}
Thus, for every $j, r$, we may choose $U_r^j \subset \mathcal{P}_{\Z[i]} \cap Q_{j, N}[\rho_N^r, \rho_N^{r + 1}]$ such that $|U_r^j| = |T_r^j|$. Let
\begin{align}
S_{1, N}
:= \bigsqcup_{j, r} U_r^j.
\end{align}
By construction, for every $1 \leq j \leq M_N$ and every $r \in \mathbb{N}_0$,
\begin{align}
\big|S_{1, N} \cap Q_{j, N}[\rho_N^r, \rho_N^{r + 1}]\big|
& = \big|S_{2, N} \cap Q_{j, N}[\rho_N^r, \rho_N^{r + 1}]\big|.
\end{align}
This proves condition \ref{part:2.3b}.

It remains to verify conditions \ref{part:2.3c} and \ref{part:2.3d} for $S_{1, N}$. If $U_r^j$ and $T_r^j$ are corresponding cells, then
\begin{align}
\sum_{m \in U_r^j} \frac{1}{\mathcal{N}(m)}
& \geq \frac{1}{\rho_N} \sum_{m \in T_r^j} \frac{1}{\mathcal{N}(m)}.
\end{align}
Hence
\begin{align}
\sum_{m \in S_{1, N}} \frac{1}{\mathcal{N}(m)}
& \geq \frac{1}{\rho_N} \sum_{m \in S_{2, N}} \frac{1}{\mathcal{N}(m)} \\
& = \frac{1}{\rho_N} \left( \sum_{p \in P_{1, N}} \frac{1}{\mathcal{N}(p)} \right) \left( \sum_{q \in P_{2, N}} \frac{1}{\mathcal{N}(q)} \right) \\
& \geq \frac{L^2}{\rho_N}.
\end{align}
Since $\rho_N \rightarrow 1$, for all sufficiently large $N$,
\begin{align}
\frac{L^2}{\rho_N}
& \geq \frac{72}{\varepsilon^2}
\geq \frac{4}{\varepsilon}.
\end{align}
Lemma \ref{lem2.4} therefore gives
\begin{align}
\E_{m \in S_{1, N}}^{\log} \E_{m' \in S_{1, N}}^{\log} \Phi(m, m')
& \leq \varepsilon.
\end{align}
Together with \eqref{eq:S2-Phi-bound}, this proves condition \ref{part:2.3c}.

Finally, every element of $S_{2, N}$ has norm at most $Y_N Z_N$, and every element of $S_{1, N}$ belongs to the same radial cell as some element of $S_{2, N}$. Hence, if
\begin{align}
R_N
& := \max_{m \in S_{1, N} \cup S_{2, N}} \mathcal{N}(m),
\end{align}
then
\begin{align}
R_N
\leq \rho_N Y_N Z_N
= \rho_N N^{1/4}\exp\Big((\log N)^{3/4}(1 + \log(1/\lambda_N))^{1/4}\Big).
\end{align}
Consequently, $R_N = O(N^{1/4 + o(1)})$, and we may bound
\begin{align}
\frac{R_N^{3/2}}{\gamma_N N^{1/2}}
\leq N^{-1/8 + o(1)}
\rightarrow 0.
\end{align}
This proves condition \ref{part:2.3d} and completes the proof.
\end{proof}

\subsection{Shift Invariance}
\begin{proof}[Proof of Theorem \ref{thm:omegaavg}]
Let $(\gamma_N)$ satisfy the hypotheses of the theorem. By Lemma \ref{lem:countsec}, uniformly for $s \in [0, 2\pi)$ and $\gamma_N \leq \gamma \leq 2\pi$,
\begin{align}
|K_{s,\gamma}(N)|
& = \frac{\gamma N}{2} + O(N^{1/2}).
\end{align}
Since $\gamma_N N^{1/2}  = N^{1/2 - o(1)} \to \infty,$ we have $|K_{s,\gamma}(N)| \geq \gamma N/4 > 0$ for all sufficiently large $N$, uniformly over $s$ and $\gamma$.

Let $\mathcal{B}$ be the set of all functions $b : \mathbb{N}_0 \rightarrow \mathbb{C}$ satisfying $\|b\|_\infty \leq 1$. For $b \in \mathcal{B}$, set
\begin{align}
A_{N,s,\gamma}(b)
& := \E_{n \in K_{s,\gamma}(N)} b(\Omega(n)), &
A_{N,s,\gamma}^+(b)
& := \E_{n \in K_{s,\gamma}(N)} b(\Omega(n) + 1).
\end{align}
We first establish uniform asymptotic invariance over all such functions. Namely, we claim that
\begin{align}
\Delta_N
& := \sup_{\gamma_N \leq \gamma \leq 2\pi} \sup_{s \in [0, 2\pi)} \sup_{b \in \mathcal{B}} |A_{N,s,\gamma}^+(b) - A_{N,s,\gamma}(b)|
\longrightarrow 0.
\label{eq:uniformArithmeticShift}
\end{align}

Fix $\varepsilon \in (0, 1)$, and for all sufficiently large $N$ let $S_{1,N}$ and $S_{2,N}$ be the sets supplied by Lemma \ref{lem:prime-semiprime}. These sets are chosen independently of $s$, $\gamma$, and $b$. Write
\begin{align}
R_N
& := \max_{m \in S_{1,N} \cup S_{2,N}} \mathcal{N}(m).
\end{align}
Since $R_N \geq 1$, condition \ref{part:2.3d} of Lemma \ref{lem:prime-semiprime} gives $\gamma_N (N/R_N)^{1/2} \rightarrow \infty$ and makes both $R_N$-dependent errors in Proposition \ref{prop:TuranKubelius} tend to zero uniformly for $\gamma_N \leq \gamma \leq 2\pi$. Thus the proposition applies uniformly for all sufficiently large $N$, and all conditional averages below are well-defined.

For $b \in \mathcal{B}$, define
\begin{align}
B_{1,N,s,\gamma}(b)
& := \E_{m \in S_{1,N}}^{\log} \E_{\substack{n \in K_{s,\gamma}(N) \\ m \mid n}} b(\Omega(n) + 1), &
B_{2,N,s,\gamma}(b)
& := \E_{m' \in S_{2,N}}^{\log} \E_{\substack{n \in K_{s,\gamma}(N) \\ m' \mid n}} b(\Omega(n)).
\end{align}
Apply Proposition \ref{prop:TuranKubelius} with the probability measures
\begin{align}
\nu_{i,N}(m)
& := \frac{\one_{m \in S_{i,N}}}{\mathcal{N}(m) \displaystyle \sum_{a \in S_{i,N}} \mathcal{N}(a)^{-1}}, \qquad i \in \{1, 2\},
\end{align}
defined for $m \in \Z[i] \setminus \{0\}$. Using condition \ref{part:2.3c} of Lemma \ref{lem:prime-semiprime} and the preceding error estimates gives
\begin{align}
\limsup_{N \rightarrow \infty} \sup_{\gamma_N \leq \gamma \leq 2\pi} \sup_{s \in [0, 2\pi)} \sup_{b \in \mathcal{B}} |A_{N,s,\gamma}^+(b) - B_{1,N,s,\gamma}(b)|
& \leq \varepsilon^{1/2}, \\ \limsup_{N \rightarrow \infty} \sup_{\gamma_N \leq \gamma \leq 2\pi} \sup_{s \in [0, 2\pi)} \sup_{b \in \mathcal{B}} |A_{N,s,\gamma}(b) - B_{2,N,s,\gamma}(b)|
& \leq \varepsilon^{1/2}.
\end{align}
The supremum over $b$ is permissible because the estimate in Proposition \ref{prop:TuranKubelius} is uniform over all arithmetic functions bounded in absolute value by $1$.

For $b \in \mathcal{B}$, define $g_b : \Z[i] \setminus \{0\} \rightarrow \mathbb{C}$ by $g_b(\ell) := b(\Omega(\ell) + 2).$
Then $|g_b| \leq 1$ uniformly over $b \in \mathcal{B}$. By condition \ref{part:2.3a} of Lemma \ref{lem:prime-semiprime}, every element of $S_{1,N}$ is a Gaussian prime and every element of $S_{2,N}$ is a Gaussian semiprime. Hence complete additivity of $\Omega$ and the substitution $n = m \ell$ give
\begin{align}
B_{1,N,s,\gamma}(b)
& = \E_{m \in S_{1,N}}^{\log} \E_{\substack{\ell \in (K_{s,\gamma})_{[m]} \\ \mathcal{N}(\ell) \leq N/\mathcal{N}(m)}} g_b(\ell), &
B_{2,N,s,\gamma}(b)
& = \E_{m' \in S_{2,N}}^{\log} \E_{\substack{\ell \in (K_{s,\gamma})_{[m']} \\ \mathcal{N}(\ell) \leq N/\mathcal{N}(m')}} g_b(\ell).
\end{align}
The estimate in Lemma \ref{lem:uniformcomp} is uniform over all functions bounded in absolute value by $1$. Therefore,
\begin{align}
\lim_{N \rightarrow \infty} \sup_{\gamma_N \leq \gamma \leq 2\pi} \sup_{s \in [0, 2\pi)} \sup_{b \in \mathcal{B}} |B_{1,N,s,\gamma}(b) - B_{2,N,s,\gamma}(b)|
& = 0.
\end{align}
Combining these estimates with the triangle inequality gives
\begin{align}
\limsup_{N \rightarrow \infty} \Delta_N
& \leq 2 \varepsilon^{1/2}.
\end{align}
Since $\varepsilon$ is arbitrary, this proves \eqref{eq:uniformArithmeticShift}.

For a fixed admissible sector, write $p_k := p(K_{s,\gamma}(N), k)$. By $\ell^\infty$-$\ell^1$ duality,
\begin{align}
\sup_{b \in \mathcal{B}} |A_{N,s,\gamma}^+(b) - A_{N,s,\gamma}(b)|
& = p_0 + |p_0 - p_1| + \sum_{k \geq 1} |p_k - p_{k + 1}|.
\end{align}
Hence $\sum_{k \geq 1} |p_k - p_{k + 1}| \leq \Delta_N$ uniformly over the admissible sectors, and \eqref{eq:uniformArithmeticShift} proves the theorem.
\end{proof}
\medskip 

\begin{proof}[Proof of Theorem \ref{thm:additivesec}]
Fix $f \in C(X)$, and let $(\gamma_N)$ satisfy the hypotheses of the theorem. Let $\Delta_N$ be as in \eqref{eq:uniformArithmeticShift}, so that $\Delta_N \rightarrow 0$ by the proof of Theorem \ref{thm:omegaavg}.

For $h \in C(X)$ with $h \neq 0$ and $x \in X$, define $b_x(k) := h(T^k x)/\|h\|_\infty$ for $k \in \mathbb{N}_0$. Since $\|b_x\|_\infty \leq 1$, the definition of $\Delta_N$ gives, for all sufficiently large $N$,
\begin{align}
\sup_{\substack{\gamma_N \leq \gamma \leq 2\pi \\ s \in [0, 2\pi) \\ x \in X}} \left|\E_{n \in K_{s,\gamma}(N)} \big(h(T^{\Omega(n) + 1} x) - h(T^{\Omega(n)} x)\big)\right|
& \leq \|h\|_\infty \Delta_N
\longrightarrow 0.
\label{sec2eq2}
\end{align}
For $h = 0$, the same conclusion is immediate.

We now prove the asserted uniform convergence. Suppose, for a contradiction, that it fails. Then there exist $\eta > 0$, a sequence $N_j \rightarrow \infty$, angles $s_j \in [0, 2\pi)$, widths
$\gamma_{N_j} \leq \beta_j \leq 2\pi,$
and points $x_j \in X$ such that
\begin{align}
\left|\E_{n \in K_{s_j,\beta_j}(N_j)} f(T^{\Omega(n)} x_j) - \int_X f \, d\mu\right|
& \geq \eta
\end{align}
for every $j \in \mathbb{N}$. Define the Borel probability measures
\begin{align}
\mu_j
& := \E_{n \in K_{s_j,\beta_j}(N_j)} \delta_{T^{\Omega(n)} x_j},
\end{align}
where $\delta_z$ denotes the Dirac probability measure at $z$.

Since $X$ is compact and metrizable, the space of Borel probability measures on $X$ is sequentially compact in the weak-$*$ topology. After passing to a subsequence, we may therefore suppose that $\mu_j \xrightarrow{w^*} \nu$ for some Borel probability measure $\nu$ on $X$.

Let $h \in C(X)$. Since the chosen widths satisfy $\beta_j \geq \gamma_{N_j}$, \eqref{sec2eq2} gives
\begin{align}
\left|\int_X h \circ T \, d\mu_j - \int_X h \, d\mu_j\right|
& = \left|\E_{n \in K_{s_j,\beta_j}(N_j)} \big(h(T^{\Omega(n) + 1} x_j) - h(T^{\Omega(n)} x_j)\big)\right|
\longrightarrow 0.
\end{align}
Because $T$ is continuous, $h \circ T \in C(X)$. Passing to the weak-$*$ limit therefore gives
\begin{align}
\int_X h \circ T \, d\nu
& = \int_X h \, d\nu.
\end{align}
Thus $\nu$ is $T$-invariant. By unique ergodicity, $\nu = \mu$. It follows that
\begin{align}
\E_{n \in K_{s_j,\beta_j}(N_j)} f(T^{\Omega(n)} x_j)
& = \int_X f \, d\mu_j \\
& \longrightarrow \int_X f \, d\nu
= \int_X f \, d\mu,
\end{align}
contradicting the choice of $N_j$, $s_j$, $\beta_j$, and $x_j$. Therefore,
\begin{align}
\sup_{\gamma_N \leq \gamma \leq 2\pi} \sup_{s \in [0, 2\pi)} \sup_{x \in X} \left|\E_{n \in K_{s,\gamma}(N)} f(T^{\Omega(n)} x) - \int_X f \, d\mu\right|
& \longrightarrow 0,
\end{align}
as required.
\end{proof}

\section{Proof of the Multiplicative Results}

We begin with a characterization of pretended invariance for finitely generated actions.

\begin{lemma}\label{lem:pretendedInvariance}
Let $(Y, S)$ be a finitely generated multiplicative topological dynamical system, let $R_1, \ldots, R_d$ be its distinct generators, and set $P_j := \{p \in \mathcal{P}_{\Z[i]} : S_p = R_j\}$ for $j \in \{1, \ldots, d\}$. A Borel probability measure $\nu$ on $Y$ pretends to be invariant under $S$ if and only if it is invariant under every $R_j$ for which
\begin{align}
\sum_{p \in P_j} \frac{1}{\mathcal{N}(p)}
& = \infty.
\end{align}
\end{lemma}

\begin{proof}
Suppose that $\nu$ pretends to be invariant under $S$. Choose $P \subset \mathcal{P}_{\Z[i]}$ such that $\sum_{p \in \mathcal{P}_{\Z[i]} \setminus P} \mathcal{N}(p)^{-1} < \infty$ and $\nu$ is invariant under $S_p$ for every $p \in P$. If the reciprocal-norm sum over $P_j$ diverges, then $P_j \cap P \neq \varnothing$, since the sum over $P_j \setminus P$ converges. Choosing $p \in P_j \cap P$ gives $S_p = R_j$, so $\nu$ is invariant under $R_j$.

Conversely, suppose that $\nu$ is invariant under every $R_j$ for which the reciprocal-norm sum over $P_j$ diverges, and let $P$ be the union of these sets $P_j$. Then $\nu$ is invariant under $S_p$ for every $p \in P$. Since the sets $P_1, \ldots, P_d$ partition $\mathcal{P}_{\Z[i]}$, the complement of $P$ is a finite union of sets with convergent reciprocal-norm sums. Thus $P$ witnesses pretended invariance.
\end{proof}

We record the following special case of \cite{DLMS2024}*{Theorem E}, which supplies the full-disk convergence used below.

\begin{theorem}\label{thm:fixedmultiplicativesec}
Let $(Y, S)$ be a finitely generated and strongly uniquely ergodic multiplicative topological dynamical system, and let $\nu$ be the unique Borel probability measure on $Y$ that pretends to be invariant under $S$. Then, for every $g \in C(Y)$ and every $y \in Y$,
\begin{align}
\E_{n \in B_N} g(S_n y)
& \longrightarrow \int_Y g \, d\nu,
\qquad N \rightarrow \infty.
\end{align}
\end{theorem}

\begin{proof}
By Lemma \ref{lem:pretendedInvariance}, $\nu$ is the unique Borel probability measure invariant under every generator $R_j$ for which
\begin{align}
\sum_{p \in P_j} \frac{1}{\mathcal{N}(p)}
& = \infty.
\end{align}
Since $B_N$ is obtained by intersecting $\Z[i] \setminus {0}$ with a dilation of the unit disk, the result follows directly from \cite{DLMS2024}*{Theorem E}.
\end{proof}

\subsection{Angular Cancellation Estimates}

\begin{lemma}\label{lem:phasesep}
Let $D \geq 1$ be an integer, and set $\delta_D := 1 - \operatorname{sinc}(\pi/D) > 0$. For any $z_1, \ldots, z_D \in \mathbb{C}$ satisfying $|z_j| = 1$,
\begin{align}
\frac{1}{2\pi} \int_0^{2\pi} \min_{1 \leq j \leq D} \big(1 - \operatorname{Re}(z_j e^{i \theta})\big) \, d\theta
& \geq \delta_D.
\end{align}
\end{lemma}

\begin{proof}
Write $z_j = e^{-i \phi_j}$ with $\phi_j \in [0, 2\pi)$. After relabeling, assume $\phi_1 \leq \cdots \leq \phi_D$, allowing repeated phases. Set $\phi_{D + 1} := \phi_1 + 2\pi$ and $\ell_j := \phi_{j + 1} - \phi_j$ for $j \in \{1, \ldots, D\}$. Then each $\ell_j \in [0, 2\pi]$ and $\sum_{j = 1}^{D} \ell_j = 2\pi$.

Define $M(\theta) := \max_{1 \leq k \leq D} \cos(\theta - \phi_k)$. On the arc between consecutive phases $\phi_j$ and $\phi_{j + 1}$, a closest phase in circular distance is one of the endpoints. Since cosine is decreasing on $[0, \pi]$, for $t \in [0, \ell_j]$ we have
$M(\phi_j + t)  = \cos\big(\min\{t, \ell_j - t\}\big).$
This identity also holds when $\ell_j = 0$. Integrating over the gap therefore gives
\begin{align}
\int_{\phi_j}^{\phi_{j + 1}} M(\theta) \, d\theta
& = 2 \int_0^{\ell_j / 2} \cos t \, dt
= 2 \sin\left(\frac{\ell_j}{2}\right).
\end{align}
Using the $2\pi$-periodicity of $M$, summing over the gaps, and applying concavity of sine on $[0, \pi]$, we obtain
\begin{align}
\frac{1}{2\pi} \int_0^{2\pi} M(\theta) \, d\theta
& = \frac{1}{\pi} \sum_{j = 1}^{D} \sin\left(\frac{\ell_j}{2}\right) \\
& \leq \frac{D}{\pi} \sin\left(\frac{1}{D} \sum_{j = 1}^{D} \frac{\ell_j}{2}\right)
= \operatorname{sinc}\left(\frac{\pi}{D}\right).
\end{align}
Consequently,
\begin{align}
\frac{1}{2\pi} \int_0^{2\pi} \min_{1 \leq j \leq D} \big(1 - \operatorname{Re}(z_j e^{i \theta})\big) \, d\theta
& = 1 - \frac{1}{2\pi} \int_0^{2\pi} M(\theta) \, d\theta \\
& \geq 1 -  \operatorname{sinc}\left(\frac{\pi}{D}\right)
= \delta_D.
\end{align}
Finally, $\delta_D > 0$ because $\sin t < t$ for $t \in (0, \pi]$. This completes the proof.
\end{proof}

\begin{lemma}\label{lem:primesumlowerbd}
For every integer $D \geq 1$, set $\kappa_D := (1-\operatorname{sinc}(\pi/D))/32 $.
There exists a constant $N_D \geq 3$, depending only on $D$, with the following property. Let $a : \mathcal{P}_{\Z[i]} \rightarrow \mathbb{C}$ satisfy $|a(p)| = 1$ for every $p \in \mathcal{P}_{\Z[i]}$ and $\#\{a(p) : p \in \mathcal{P}_{\Z[i]}\} \leq D$. Then, for every $N \geq N_D$,
\begin{align}
\frac{1}{4} \sum_{\substack{p \in \mathcal{P}_{\Z[i]} \\ \mathcal{N}(p) \leq N}} \frac{1 - \operatorname{Re}\big(a(p) e^{i h \arg(p)} \mathcal{N}(p)^{-i t}\big)}{\mathcal{N}(p)}
& \geq \kappa_D \log \log N,
\end{align}
uniformly over all such $a$, all $h \in \Z$ satisfying $1 \leq |h| \leq \lfloor(\log N)^2\rfloor$, and all $t \in \R$ satisfying $|t| \leq \log N$.
\end{lemma}

\begin{proof}
Fix $D \geq 1$, and set $\delta_D := 1 - \operatorname{sinc}(\pi/D) > 0$. List the values of $a$ as $z_1, \ldots, z_D$, repeating values if necessary, and define
\begin{align}
G(v) & := \min_{1 \leq j \leq D} 
    \big(1 - \operatorname{Re}(z_j e^{i v})\big).
\end{align}
The function $G$ is $2\pi$-periodic, takes values in $[0, 2]$, and is $1$-Lipschitz. By Lemma \ref{lem:phasesep}, periodicity, and a change of variables, for every nonzero integer $h$ and every $\beta \in \R$,
\begin{align}
\frac{1}{2\pi} \int_0^{2\pi} G(h \theta + \beta) \, d\theta
& = \frac{1}{2\pi} \int_0^{2\pi} G(v) \, dv
\geq \delta_D.
\label{eq:prime-distance-angular-mean}
\end{align}

For $N$ sufficiently large, put $L := \log N$, $Q := L^{100}$, $M := 4 \lceil L^4\rceil$, and $\omega := 2\pi/M$. In particular, $Q < N$. All estimates below are uniform over the admissible choices of $a$, $h$, and $t$.

For $j \in \{0, \ldots, M - 1\}$, set $I_j := [j \omega, (j + 1) \omega]$ and $J_j := [(j + 1/4) \omega, (j + 3/4) \omega)$. Thus $J_j$ is the middle half of $I_j$. Since $M$ is divisible by $4$, the closure of each $J_j$ lies in the interior of one quadrant.

For $Q \leq x \leq N$, define $\mathcal{P}_j(x) := \{p \in \mathcal{P}_{\Z[i]} : x - x^{9/10} < \mathcal{N}(p) \leq x,\ \arg(p) \in J_j\}$. The angular width of $J_j$ is $\omega/2$, and
\begin{align}
\frac{(\omega/2) x^{9/10}}{x^{7/10 + 1/20}}
& = \frac{\omega}{2} x^{3/20}
\gtrsim L^{-4} Q^{3/20}
= L^{11}
\longrightarrow \infty.
\end{align}
Therefore Theorem \ref{thm:stucky}, with the fixed choices $\theta = 9/10$ and $\zeta = 1/20$, applies uniformly over $Q \leq x \leq N$ and all $j$, after rotating by a unit when necessary.

We remove the contribution of prime powers of exponent at least two from the weighted sum in that theorem. Their total contribution up to norm $x$ is $O(x^{1/2}(\log x)^2)$: there are $O(\log x)$ possible exponents, for each exponent at least two there are $O(x^{1/2})$ possible underlying prime ideals, and each weight is at most $\log x$. The ideal-counting bound follows from Lemma \ref{lem:countsec}, since every nonzero ideal is principal and has exactly four generators. Using $(\log x)^2 \lesssim x^{1/10}$, we have
\begin{align}
\frac{x^{1/2}(\log x)^2}{\omega x^{9/10}}
& \lesssim L^4 x^{-3/10}
\leq L^4 Q^{-3/10}
= L^{-26}
\longrightarrow 0.
\end{align}
Hence this contribution is negligible uniformly in the stated range. Since each $J_j$ has width less than $\pi/2$, each prime ideal has at most one generator with argument in $J_j$. The main term in Theorem \ref{thm:stucky} is $\omega x^{9/10}/\pi$, so, for all sufficiently large $N$,
\begin{align}
\#\mathcal{P}_j(x)
& \geq \frac{1}{\log x} \sum_{p \in \mathcal{P}_j(x)} \log \mathcal{N}(p)
\geq \frac{\omega x^{9/10}}{2\pi \log x},
\label{eq:prime-distance-cell-count}
\end{align}
uniformly over $Q \leq x \leq N$ and all $j$.

Fix an integer $h$ with $1 \leq |h| \leq \lfloor L^2\rfloor$ and a real number $t$ with $|t| \leq L$. For every $p \in \mathcal{P}_{\Z[i]}$, write
\begin{align}
A(p)
& := 1 - \operatorname{Re}\big(a(p) e^{i h \arg(p)} \mathcal{N}(p)^{-i t}\big).
\end{align}
By the definition of $G$, $A(p) \geq G\big(h \arg(p) - t \log \mathcal{N}(p)\big)
\geq 0.$

Set $\varepsilon_N := 2 L^{-9}$. If $Q \leq x \leq N$ and $x - x^{9/10} < \mathcal{N}(p) \leq x$,  for sufficiently large $N$,
\begin{align}
|t| \big|\log \mathcal{N}(p) - \log x\big|
& \leq -L \log(1 - x^{-1/10})
\leq 2 L x^{-1/10}
\leq 2 L Q^{-1/10}
= \varepsilon_N.
\end{align}
For each $j$ and $x$, define
\begin{align}
b_j(x)
& := \max\left\{0, \inf_{\theta \in I_j} G(h \theta - t \log x) - \varepsilon_N\right\}.
\end{align}
The Lipschitz bound for $G$ shows that $A(p) \geq b_j(x)$ whenever $p \in \mathcal{P}_j(x)$.

Moreover, $\theta \mapsto G(h \theta - t \log x)$ is $|h|$-Lipschitz. Integrating over each $I_j$ and using \eqref{eq:prime-distance-angular-mean}, we obtain
\begin{align}
\omega \sum_{j = 0}^{M - 1} b_j(x)
& \geq \int_0^{2\pi} G(h \theta - t \log x) \, d\theta - 2\pi \big(|h| \omega + \varepsilon_N\big) \\
& \geq 2\pi \delta_D - 2\pi \big(L^2 \omega + \varepsilon_N\big)
\geq \pi \delta_D
\end{align}
for all sufficiently large $N$ depending only on $D$, since $L^2 \omega + \varepsilon_N \rightarrow 0$.

The sets $\mathcal{P}_j(x)$ are pairwise disjoint, and all $A(p)$ are nonnegative. Consequently, \eqref{eq:prime-distance-cell-count} gives
\begin{align}
S(x)
& := \frac{1}{4} \sum_{\substack{p \in \mathcal{P}_{\Z[i]} \\ x - x^{9/10} < \mathcal{N}(p) \leq x}} A(p) \\
& \geq \frac{1}{4} \sum_{j = 0}^{M - 1} b_j(x) \#\mathcal{P}_j(x) \\
& \geq \frac{x^{9/10}}{8\pi \log x}\,\omega \sum_{j = 0}^{M - 1} b_j(x)
\geq \frac{\delta_D}{8} \frac{x^{9/10}}{\log x},
\label{eq:prime-distance-short-annulus}
\end{align}
uniformly over $Q \leq x \leq N$.

For $m > 0$, define
\begin{align}
W_N(m)
& := \int_Q^N \frac{\one_{x - x^{9/10} < m \leq x}}{x^{19/10}} \, dx.
\end{align}
For a contributing $x$, we have $x^{9/10} \leq x/2$ once $N$ is sufficiently large. Thus $m \leq x < 2 m$ and $0 \leq x - m < (2 m)^{9/10}$. It follows that
\begin{align}
W_N(m)
& \leq \frac{(2 m)^{9/10}}{m^{19/10}}
\leq \frac{2}{m}.
\end{align}
Integrating \eqref{eq:prime-distance-short-annulus} against $x^{-19/10}\,dx$ and interchanging the finite sum and the integral gives
\begin{align}
\frac{\delta_D}{8} \int_Q^N \frac{dx}{x \log x}
& \leq \int_Q^N \frac{S(x)}{x^{19/10}} \, dx \\
& = \frac{1}{4} \sum_{\substack{p \in \mathcal{P}_{\Z[i]} \\ \mathcal{N}(p) \leq N}} A(p) W_N(\mathcal{N}(p)) \\
& \leq 2 \left(\frac{1}{4} \sum_{\substack{p \in \mathcal{P}_{\Z[i]} \\ \mathcal{N}(p) \leq N}} \frac{A(p)}{\mathcal{N}(p)}\right).
\end{align}
Since $Q = (\log N)^{100}$, we conclude that
\begin{align}
\frac{1}{4} \sum_{\substack{p \in \mathcal{P}_{\Z[i]} \\ \mathcal{N}(p) \leq N}} \frac{A(p)}{\mathcal{N}(p)}
& \geq \frac{\delta_D}{16} \log\left(\frac{\log N}{\log Q}\right) \\
& = \frac{\delta_D}{16} \big(\log \log N - \log(100 \log \log N)\big) \\
& \geq \frac{\delta_D}{32} \log \log N
\end{align}
for all sufficiently large $N$.

Every threshold used above depends at most on $D$, and none depends on the values or assignment of $a$, or on the admissible parameters $h$ and $t$. Taking $\kappa_D := \delta_D/32$ and choosing $N_D$ sufficiently large completes the proof.
\end{proof}

\begin{lemma}\label{lem:halaszbd}
Let $a : \Z[i] \setminus \{0\} \rightarrow \mathbb{C}$ be completely multiplicative, with $|a(n)| = 1$ for every $n \neq 0$. Let $h \in \Z$ and $N \geq 3$.

If $a(i) i^h \neq 1$, then
\begin{align}
\E_{n \in B_N} a(n) e^{i h \arg(n)}
& = 0.
\end{align}
If $a(i) i^h = 1$, define
\begin{align}
M_{a,h}(N)
& := \min_{|t| \leq \log N} \frac{1}{4} \sum_{\substack{p \in \mathcal{P}_{\Z[i]} \\ \mathcal{N}(p) \leq N}} \frac{1 - \operatorname{Re}\big(a(p) e^{i h \arg(p)} \mathcal{N}(p)^{-i t}\big)}{\mathcal{N}(p)}.
\end{align}
Then
\begin{align}
\left|\E_{n \in B_N} a(n) e^{i h \arg(n)}\right|
& \lesssim \big(1 + M_{a,h}(N)\big) e^{-M_{a,h}(N)} + \frac{\log \log N}{\log N},
\end{align}
with an absolute implied constant.
\end{lemma}

\begin{proof}
Reindexing the disk sum by $n \mapsto i n$ proves the first assertion. In the second case, the function $n \mapsto a(n) e^{i h \arg(n)}$ is invariant under associates and therefore defines a completely multiplicative function on the nonzero ideals of $\Z[i]$. Each such ideal has exactly four generators, so its prime-distance parameter is precisely $M_{a,h}(N)$. Apply \cite{Kus2026}*{Theorem 1.1} and use $|B_N| \simeq N$, which follows from Lemma \ref{lem:countsec}.
\end{proof}

\subsection{Sector-Disk Comparison}
We now combine the auxiliary lemmas to obtain the following angular cancellation estimate, which is the main arithmetic input in the proof of Theorem \ref{thm:sector-disk}.

\begin{proposition}\label{prop:angularcancellation}
For every integer $D \geq 1$, set $c_D := \big(1 - \operatorname{sinc}(\pi/D)\big)/64$.
There exist constants $C_D > 0$ and $N_D \geq 3$, depending only on $D$, with the following property. Let $a : \Z[i] \setminus \{0\} \rightarrow \mathbb{C}$ be completely multiplicative, with $|a(n)| = 1$ for every $n \neq 0$, and suppose 
\begin{align}
\#\{a(p) : p \in \mathcal{P}_{\Z[i]}\}
& \leq D.
\end{align}
Then, for every $N \geq N_D$,
\begin{align}
\max_{\substack{h \in \Z \\ 1 \leq |h| \leq \lfloor(\log N)^2\rfloor}} \left|\E_{n \in B_N} a(n) e^{i h \arg(n)}\right|
& \leq C_D (\log N)^{-c_D}.
\end{align}
\end{proposition}

\begin{proof}
Let $\kappa_D$ be as in Lemma \ref{lem:primesumlowerbd}, and fix an integer $h$ with $1 \leq |h| \leq \lfloor(\log N)^2\rfloor$. If $a(i) i^h \neq 1$, the average vanishes by Lemma \ref{lem:halaszbd}. Otherwise, Lemma \ref{lem:primesumlowerbd} gives $M_{a,h}(N) \geq \kappa_D \log \log N$ for all sufficiently large $N$ depending only on $D$. Lemma \ref{lem:halaszbd} then yields
\begin{align}
\left|\E_{n \in B_N} a(n) e^{i h \arg(n)}\right|
& \lesssim_D (1 + \log \log N) (\log N)^{-\kappa_D} + \frac{\log \log N}{\log N} \\
& \lesssim_D (\log N)^{-c_D},
\end{align}
since $c_D = \kappa_D/2$ and $0 < \kappa_D \leq 1/32$. All bounds are uniform over the admissible functions $a$ and integers $h$, proving the proposition.
\end{proof}

\begin{proof}[Proof of Theorem \ref{thm:sector-disk}]
Let $d_0 := \#\{S_p : p \in \mathcal{P}_{\Z[i]}\}$, so that $1 \leq d_0 \leq d$, and enumerate the distinct prime transformations as $R_1, \ldots, R_{d_0}$. Put $D := 4d$ and $c := c_D = 2 A_d$, where $c_D$ is the exponent defined in Proposition \ref{prop:angularcancellation}. In particular, $0 < c \leq 1/2$ and $0 < A_d \leq 1/4$. For $N \geq 3$, write $H_N := \lfloor(\log N)^2\rfloor.$
Throughout the proof, the admissible sectors are those satisfying $s \in [0, 2\pi)$ and $(\log N)^{-A_d} \leq \gamma \leq 2\pi$.

Let $\mathcal{A}_D$ denote the collection of completely multiplicative functions $a : \Z[i] \setminus \{0\} \rightarrow \mathbb{C}$ satisfying $|a(n)| = 1$ for every $n \neq 0$ and taking at most $D$ distinct values on $\mathcal{P}_{\Z[i]}$. Proposition \ref{prop:angularcancellation} and the choice $c \leq c_D$ give
\begin{align}
\sup_{a \in \mathcal{A}_D} \max_{\substack{h \in \Z \\ 1 \leq |h| \leq H_N}} \left|\E_{n \in B_N} a(n) e^{i h \arg(n)}\right|
& \lesssim_D (\log N)^{-c}.
\label{eq:multi-angular-cancellation}
\end{align}

For an integer $H \geq 1$, let
\begin{align}
\mathcal{F}_H(t)
& := \frac{1}{H + 1} \left|\sum_{k = 0}^{H} e^{i k t}\right|^2.
\end{align}
This normalized Fej\'er kernel is nonnegative, has $m_{\mathrm{ang}}$-integral $1$, and satisfies
\begin{align}
\int_{\mathbb{T}_{\mathrm{ang}}} |t| \mathcal{F}_H(t) \, dm_{\mathrm{ang}}(t)
& \lesssim \frac{\log(H + 2)}{H + 1},
\label{eq:multi-fejer-moment}
\end{align}
where $t$ is represented in $[-\pi, \pi]$. For $I := [s, s + \gamma)$, set $\alpha_I := \gamma/(2\pi)$ and $P_{I,H} := \mathcal{F}_H * \one_I$, with convolution taken with respect to $m_{\mathrm{ang}}$. This polynomial has degree at most $H$, and its Fourier coefficients satisfy
\begin{align}
\widehat{P}_{I,H}(0)
& = \alpha_I, \qquad \sum_{1 \leq |h| \leq H} |\widehat{P}_{I,H}(h)|
\lesssim \log(H + 2).
\label{eq:multi-fejer-coefficients}
\end{align}

The symmetric difference $I \mathbin{\triangle} (I + t)$ is a union of at most two angular intervals with total length at most $2|t|$. Lemma \ref{lem:countsec}, applied to those intervals and to the full disk, therefore gives
\begin{align}
\E_{n \in B_N} \left|\one_I(\arg(n)) - \one_I(\arg(n) - t)\right|
& \lesssim |t| + N^{-1/2},
\end{align}
uniformly over $I$ and $t \in [-\pi, \pi]$. Positivity of $\mathcal{F}_H$ and \eqref{eq:multi-fejer-moment} imply
\begin{align}
\E_{n \in B_N} & \left|\one_I(\arg(n)) - P_{I,H}(\arg(n))\right| \\
& \leq \int_{\mathbb{T}_{\mathrm{ang}}} \mathcal{F}_H(t) \E_{n \in B_N} \left|\one_I(\arg(n)) - \one_I(\arg(n) - t)\right| \, dm_{\mathrm{ang}}(t) \\
& \lesssim \frac{\log(H + 2)}{H + 1} + N^{-1/2}.
\label{eq:multi-fejer-error}
\end{align}

Take $H = H_N$. For every $a \in \mathcal{A}_D$, the Fourier expansion of $P_{I,H_N}$, together with \eqref{eq:multi-angular-cancellation}, \eqref{eq:multi-fejer-coefficients}, and \eqref{eq:multi-fejer-error}, gives
\begin{align}
\left|\E_{n \in B_N} \one_I(\arg(n)) a(n) - \alpha_I \E_{n \in B_N} a(n)\right|
& \lesssim_D (\log N)^{-c} \log \log N + \frac{\log \log N}{(\log N)^2} + N^{-1/2} \\
& \lesssim_D (\log N)^{-c} \log \log N.
\label{eq:multi-unnormalized-sector-comparison}
\end{align}

By Lemma \ref{lem:countsec},
\begin{align}
|B_N|
& = \pi N + O(N^{1/2}), \\ |K_{s,\gamma}(N)|
& = \frac{\gamma N}{2} + O(N^{1/2}),
\end{align}
uniformly over $s$ and $\gamma$. Since $(\log N)^{-A_d} N^{1/2} \rightarrow \infty$, all admissible sectors are nonempty for sufficiently large $N$. Moreover, writing
\begin{align}
q_{N,I}
& := \frac{|K_{s,\gamma}(N)|}{|B_N|},
\end{align}
we have $q_{N,I} = \alpha_I + O(N^{-1/2}) \gtrsim \gamma$
uniformly over the admissible sectors.

Consequently, \eqref{eq:multi-unnormalized-sector-comparison} implies
\begin{align}
\left|\E_{n \in K_{s,\gamma}(N)} a(n) - \E_{n \in B_N} a(n)\right|
& = \frac{1}{q_{N,I}} \left|\E_{n \in B_N} \one_I(\arg(n)) a(n) - q_{N,I} \E_{n \in B_N} a(n)\right| \\
& \lesssim_D \frac{1}{\gamma} \left((\log N)^{-c} \log \log N + N^{-1/2}\right).
\end{align}
Set $\eta_N := (\log N)^{A_d - c} \log \log N = (\log N)^{-c/2} \log \log N.$
We have proved
\begin{align}
\sup_{\substack{(\log N)^{-A_d} \leq \gamma \leq 2\pi \\ s \in [0, 2\pi)}} \sup_{a \in \mathcal{A}_D} \left|\E_{n \in K_{s,\gamma}(N)} a(n) - \E_{n \in B_N} a(n)\right|
& \lesssim_d \eta_N.
\label{eq:multi-scalar-sector-disk}
\end{align}

Choose a set $\mathcal{P}^* \subset \mathcal{P}_{\Z[i]}$ containing exactly one representative from each associate class, and set $P_j^* := \{p \in \mathcal{P}^* : S_p = R_j\}$ for $j \in \{1, \ldots, d_0\}$. The sets $P_j^*$ are pairwise disjoint and cover $\mathcal{P}^*$; some of them may be empty. These choices are fixed independently of $N$. Unique factorization gives
\begin{align}
n
& = i^{r(n)} \prod_{p \in \mathcal{P}^*} p^{v_p(n)},
\end{align}
where $r(n) \in \Z/4\Z$ and $v_p(n) \in \mathbb{N}_0$. Define
\begin{align}
\Omega_j(n)
& := \sum_{p \in P_j^*} v_p(n), \\ C(n)
& := \big(r(n), \Omega_1(n), \ldots, \Omega_{d_0}(n)\big).
\end{align}
Then $C(n)$ takes values in $\mathcal{C}_{d_0} := \Z/4\Z \times \mathbb{N}_0^{d_0}$, and
\begin{align}
\sum_{j = 1}^{d_0} \Omega_j(n)
& = \Omega(n).
\label{eq:multi-total-prime-count}
\end{align}
Since the transformations in the action commute, multiplicativity gives
\begin{align}
S_n
& = S_{i^{r(n)}} \circ R_1^{\Omega_1(n)} \circ \cdots \circ R_{d_0}^{\Omega_{d_0}(n)}.
\label{eq:multi-action-factorization}
\end{align}

For $\mathbf{z} = (z_1, \ldots, z_{d_0}) \in (\mathbb{S}^1)^{d_0}$ and $\ell \in \{0, 1, 2, 3\}$, define
\begin{align}
a_{\mathbf{z},\ell}(n)
& := i^{\ell r(n)} \prod_{j = 1}^{d_0} z_j^{\Omega_j(n)}.
\end{align}
The functions $r$ and $\Omega_j$ are additive, with $r$ interpreted modulo $4$, so each $a_{\mathbf{z},\ell}$ is completely multiplicative and has absolute value $1$. Every Gaussian prime is of the form $i^r p$ for some $p \in P_j^*$, and $a_{\mathbf{z},\ell}(i^r p) = i^{\ell r} z_j.$
Thus $a_{\mathbf{z},\ell}$ takes at most $4d_0$ values on $\mathcal{P}_{\Z[i]}$ and belongs to $\mathcal{A}_D$, since $4d_0 \leq 4d = D$. Applying \eqref{eq:multi-scalar-sector-disk} gives
\begin{align}
\sup_{\substack{(\log N)^{-A_d} \leq \gamma \leq 2\pi \\ s \in [0, 2\pi) \\ \mathbf{z} \in (\mathbb{S}^1)^{d_0} \\ \ell \in \{0, 1, 2, 3\}}} \left|\E_{n \in K_{s,\gamma}(N)} a_{\mathbf{z},\ell}(n) - \E_{n \in B_N} a_{\mathbf{z},\ell}(n)\right|
& \lesssim_d \eta_N.
\label{eq:multi-character-comparison}
\end{align}

We first control the contribution of large prime-factor counts. The Gaussian prime-counting upper bound, recalled in \cite{DLMS2024}*{Section 2.2}, gives
\begin{align}
\#\{p \in \mathcal{P}^* : \mathcal{N}(p)
\leq t\}
& \lesssim \frac{t}{\log t}
\end{align}
for $t \geq 3$. Partial summation, treating $r = 1$, $r = 2$, and $r \geq 3$ separately, gives
\begin{align}
\sum_{\substack{p \in \mathcal{P}^*,\, r \geq 1 \\ \mathcal{N}(p)^r \leq N}} \mathcal{N}(p)^{-r}
& \lesssim \log \log N, \\ \sum_{\substack{p \in \mathcal{P}^*,\, r \geq 1 \\ \mathcal{N}(p)^r \leq N}} \mathcal{N}(p)^{-r/2}
& \lesssim \frac{N^{1/2}}{\log N} + \log \log N.
\label{eq:multi-prime-power-sums}
\end{align}

For every $n \neq 0$,
\begin{align}
\Omega(n)
& = \sum_{p \in \mathcal{P}^*} \sum_{r \geq 1} \one_{p^r \mid n}.
\end{align}
Applying Lemma \ref{lem:divisorest} and then \eqref{eq:multi-prime-power-sums}, we obtain
\begin{align}
\E_{n \in K_{s,\gamma}(N)} \Omega(n)
& = \sum_{\substack{p \in \mathcal{P}^*,\, r \geq 1 \\ \mathcal{N}(p)^r \leq N}} \mathcal{N}(p)^{-r} + O\left(\frac{1}{\gamma N^{1/2}} \sum_{\substack{p \in \mathcal{P}^*,\, r \geq 1 \\ \mathcal{N}(p)^r \leq N}} \mathcal{N}(p)^{-r/2}\right) \\
& \lesssim \log \log N + \frac{1}{\gamma \log N} + \frac{\log \log N}{\gamma N^{1/2}} \\
& \lesssim_d \log \log N,
\end{align}
uniformly over the admissible sectors, where the last inequality uses $A_d \leq 1/4$. This also applies to $B_N$, by taking $s = 0$ and $\gamma = 2\pi$. Markov's inequality therefore gives, for every $L \geq 1$,
\begin{align}
\sup_{\substack{(\log N)^{-A_d} \leq \gamma \leq 2\pi \\ s \in [0, 2\pi)}} \frac{\big|\{n \in K_{s,\gamma}(N) : \Omega(n)
> L\}\big|}{|K_{s,\gamma}(N)|} + \frac{\big|\{n \in B_N : \Omega(n)
> L\}\big|}{|B_N|}
& \lesssim_d \frac{\log \log N}{L}. \- \- \qquad \-
\label{eq:multi-prime-count-tail}
\end{align}

For any finite nonempty set $E \subset \Z[i] \setminus \{0\}$, let $p_E$ denote the distribution of $C(n)$ under uniform averaging over $E$:
\begin{align}
p_E(r, \mathbf{k})
& := \frac{\big|\{n \in E : C(n)
= (r, \mathbf{k})\}\big|}{|E|},
\end{align}
where $r \in \Z/4\Z$ and $\mathbf{k} = (k_1, \ldots, k_{d_0}) \in \mathbb{N}_0^{d_0}$. Let $m_{d_0}$ denote normalized Haar measure on $(\mathbb{S}^1)^{d_0}$. Fourier inversion gives
\begin{align}
p_E(r, \mathbf{k})
& = \frac{1}{4} \sum_{\ell = 0}^{3} i^{-\ell r} \int_{(\mathbb{S}^1)^{d_0}} \left(\E_{n \in E} a_{\mathbf{z},\ell}(n)\right) \prod_{j = 1}^{d_0} z_j^{-k_j} \, dm_{d_0}(\mathbf{z}).
\end{align}
Combining this identity with \eqref{eq:multi-character-comparison} gives
\begin{align}
\left|p_{K_{s,\gamma}(N)}(r, \mathbf{k}) - p_{B_N}(r, \mathbf{k})\right|
& \lesssim_d \eta_N
\label{eq:multi-individual-count-comparison}
\end{align}
uniformly over the admissible sectors and all $(r, \mathbf{k}) \in \mathcal{C}_{d_0}$.

For an integer $L \geq 1$, set $\mathcal{C}_{d_0}(L) := \Z/4\Z \times \{0, \ldots, L\}^{d_0}.$
This set has $4(L + 1)^{d_0}$ elements, which is at most $4(L + 1)^d$. If $C(n) \notin \mathcal{C}_{d_0}(L)$, then \eqref{eq:multi-total-prime-count} implies $\Omega(n) > L$. Summing \eqref{eq:multi-individual-count-comparison} over $\mathcal{C}_{d_0}(L)$ and using \eqref{eq:multi-prime-count-tail} outside this set gives
\begin{align}
\sup_{\substack{(\log N)^{-A_d} \leq \gamma \leq 2\pi \\ s \in [0, 2\pi)}} \sum_{\mathbf{c} \in \mathcal{C}_{d_0}} \left|p_{K_{s,\gamma}(N)}(\mathbf{c}) - p_{B_N}(\mathbf{c})\right|
& \lesssim_d (L + 1)^d \eta_N + \frac{\log \log N}{L}.
\end{align}
Since $\eta_N = (\log N)^{-A_d} \log \log N$, we balance the two terms in the preceding estimate by taking $L  := \left\lfloor(\log N)^{A_d/(d + 1)}\right\rfloor.$
For all sufficiently large $N$, this gives
\begin{align}
(L + 1)^d \eta_N + \frac{\log \log N}{L}
& \lesssim_d \frac{\log \log N}{(\log N)^{A_d/(d + 1)}}.
\end{align}
Set $b_d := A_d/(2(d + 1))$. Since $\log \log N \lesssim_d (\log N)^{b_d}$, we conclude that
\begin{align}
\mathcal{D}_N
& := \sup_{\substack{(\log N)^{-A_d} \leq \gamma \leq 2\pi \\ s \in [0, 2\pi)}} \sum_{\mathbf{c} \in \mathcal{C}_{d_0}} \left|p_{K_{s,\gamma}(N)}(\mathbf{c}) - p_{B_N}(\mathbf{c})\right| \\
& \lesssim_d (\log N)^{-b_d}.
\end{align}

Fix $g \in C(Y)$. For each $y \in Y$, define $F_y : \mathcal{C}_{d_0} \rightarrow \mathbb{C}$ by
\begin{align}
F_y(r, \mathbf{k})
& := g\big((S_{i^r} \circ R_1^{k_1} \circ \cdots \circ R_{d_0}^{k_{d_0}})(y)\big).
\end{align}
By \eqref{eq:multi-action-factorization}, we have $g(S_n y) = F_y(C(n))$, and $\|F_y\|_\infty \leq \|g\|_\infty$ uniformly in $y$. Hence
\begin{align}
\sup_{\substack{(\log N)^{-A_d} \leq \gamma \leq 2\pi \\ s \in [0, 2\pi) \\ y \in Y}} \left|\E_{n \in K_{s,\gamma}(N)} g(S_n y) - \E_{n \in B_N} g(S_n y)\right|
& \leq \|g\|_\infty \mathcal{D}_N
\lesssim_d \|g\|_\infty (\log N)^{-b_d}.
\end{align}
All implied constants and lower thresholds for $N$ depend only on $d$. Choosing $C_d$ and $N_d$ accordingly proves the theorem.
\end{proof}

\begin{proof}[Proof of Theorem \ref{thm:multisec}]
Fix $g \in C(Y)$ and $y \in Y$. Applying Theorem \ref{thm:fixedmultiplicativesec} to the full sector gives
\begin{align}
\E_{n \in B_N} g(S_n y)
& \longrightarrow \int_Y g \, d\nu.
\end{align}
Theorem \ref{thm:sector-disk} tells us that  the sector averages differ from these full-disk averages by at most $C_d \|g\|_\infty (\log N)^{-b_d}$, uniformly over the sectors in the theorem statement. Since $b_d > 0$, the triangle inequality proves the asserted convergence.
\end{proof}

\section{Arithmetic and Dynamical Consequences}
\subsection{Equi-distribution of \texorpdfstring{$\Omega$}{Omega}}
Finite cyclic rotations give equidistribution of $\Omega(n)$ modulo every fixed integer, including cancellation of the Gaussian Liouville function.

\begin{corollary}\label{implication1}
Let $(\gamma_N)_{N \geq 3}$ be a sequence of positive real numbers satisfying $\gamma_N \rightarrow 0$ and $\gamma_N^{-1} = N^{o(1)}$. Then, for every fixed integer $q \geq 2$,
\begin{align}
\sup_{\substack{\gamma_N \leq \gamma \leq 2\pi \\ s \in [0, 2\pi)}} \max_{a \in \{0, \ldots, q - 1\}} \left|\frac{\big|\{n \in K_{s,\gamma}(N) : \Omega(n)
\equiv a \pmod q\}\big|}{|K_{s,\gamma}(N)|} - \frac{1}{q}\right|
& \longrightarrow 0,
\qquad 
N \rightarrow \infty.
\end{align}
In particular, the Gaussian Liouville function $\lambda(n) := (-1)^{\Omega(n)}$ satisfies
\begin{align}
\sup_{\substack{\gamma_N \leq \gamma \leq 2\pi \\ s \in [0, 2\pi)}} \left|\E_{n \in K_{s,\gamma}(N)} \lambda(n)\right|
& \longrightarrow 0,
\qquad 
N \rightarrow \infty.
\end{align}
\end{corollary}
\begin{proof}
Apply Theorem \ref{thm:additivesec} to the cyclic rotation $x \mapsto x + 1$ on $\Z/q\Z$, with its uniform invariant measure, initial point $0$, and singleton indicators. Taking the maximum over the finitely many residue classes gives the first assertion. The second follows by taking $q = 2$ and using $\lambda(n) = 2 \one_{\Omega(n) \in 2\Z} - 1$.
\end{proof}

For $\alpha \in \R \setminus \mathbb{Q}$, applying Theorem \ref{thm:additivesec} to the irrational rotation $R_\alpha(y) := y + \alpha$ on $\mathbb{T}$ gives, for every sequence $(\gamma_N)$ satisfying its hypotheses and every $F \in C(\mathbb{T})$,
\begin{align}
\sup_{\substack{\gamma_N \leq \gamma \leq 2\pi \\ s \in [0, 2\pi) \\ y \in \mathbb{T}}} \left|\E_{n \in K_{s,\gamma}(N)} F\big(y + \Omega(n) \alpha\big) - \int_{\mathbb{T}} F \, dm_{\mathbb{T}}\right|
& \longrightarrow 0,
\qquad 
N \rightarrow \infty.
\label{eq:irrational-rotation-consequence}
\end{align}

As a consequence, we obtain equidistribution modulo one uniformly over moving and shrinking sectors. The conclusion is also uniform over the initial point and the target interval.

\begin{corollary}
Fix $\alpha \in \R \setminus \mathbb{Q}$, and let $(\gamma_N)_{N \geq 3}$ be a sequence of positive real numbers satisfying $\gamma_N \rightarrow 0$ and $\gamma_N^{-1} = N^{o(1)}$. Let $\mathcal{I}$ denote the collection of all intervals in $\mathbb{T}$. Then
\begin{align}
\sup_{\substack{\gamma_N \leq \gamma \leq 2\pi \\ s \in [0, 2\pi) \\ y \in \mathbb{T} \\ I \in \mathcal{I}}} \left|\frac{\big|\{n \in K_{s,\gamma}(N) : y + \Omega(n) \alpha \in I\}\big|}{|K_{s,\gamma}(N)|} - m_{\mathbb{T}}(I)\right|
& \longrightarrow 0,
\qquad 
N \rightarrow \infty
\end{align}
and the convergence in \eqref{eq:irrational-rotation-consequence} holds for every Riemann-integrable function $F : \mathbb{T} \rightarrow \mathbb{C}$.
\end{corollary}

\begin{proof}
By Lemma \ref{lem:countsec} and $\gamma_N N^{1/2} \rightarrow \infty$, the averaging sets are nonempty for all sufficiently large $N$, uniformly over the admissible sectors. For such $N$, write
\begin{align}
P_{N,s,\gamma,y}(I)
& := \frac{\big|\{n \in K_{s,\gamma}(N) : y + \Omega(n) \alpha \in I\}\big|}{|K_{s,\gamma}(N)|}.
\end{align}

We first establish uniform convergence for each fixed half-open interval $J \subset \mathbb{T}$. Fix $\eta > 0$. Since the boundary of $J$ has Haar measure zero, there exist real-valued functions $F_J^-, F_J^+ \in C(\mathbb{T})$ satisfying $0 \leq F_J^- \leq \one_J \leq F_J^+ \leq 1$ and
\begin{align}
\int_{\mathbb{T}} \big(F_J^+ - F_J^-\big) \, dm_{\mathbb{T}}
& \leq \eta.
\end{align}
Applying \eqref{eq:irrational-rotation-consequence} to these two functions and using the pointwise inequalities gives
\begin{align}
\limsup_{N \rightarrow \infty} \sup_{\substack{\gamma_N \leq \gamma \leq 2\pi \\ s \in [0, 2\pi) \\ y \in \mathbb{T}}} \left|P_{N,s,\gamma,y}(J) - m_{\mathbb{T}}(J)\right|
& \leq \eta.
\end{align}
Since $\eta$ is arbitrary, this proves uniform convergence for every fixed half-open interval.

To obtain uniformity over all intervals, fix an integer $q \geq 2$. Let $\mathcal{I}_q$ be the finite family of half-open intervals in $\mathbb{T}$ whose endpoints belong to $\{0, 1/q, \ldots, (q - 1)/q\}$, including intervals that cross $0$, together with the empty set and $\mathbb{T}$. By the preceding convergence and finiteness of $\mathcal{I}_q$,
\begin{align}
E_N(q)
& := \max_{J \in \mathcal{I}_q} \sup_{\substack{\gamma_N \leq \gamma \leq 2\pi \\ s \in [0, 2\pi) \\ y \in \mathbb{T}}} \left|P_{N,s,\gamma,y}(J) - m_{\mathbb{T}}(J)\right|
\longrightarrow 0
\end{align}
as $N \rightarrow \infty$ with $q$ fixed.

Every interval $I \in \mathcal{I}$ admits intervals $J^-, J^+ \in \mathcal{I}_q$ such that $J^- \subseteq I \subseteq J^+$ and
\begin{align}
m_{\mathbb{T}}(J^+) - m_{\mathbb{T}}(J^-)
& \leq \frac{2}{q}.
\end{align}
Indeed, take $J^-$ to be the union of the grid cells contained in $I$ and $J^+$ to be the union of the grid cells meeting $I$. These unions are intervals, possibly empty or equal to $\mathbb{T}$, and they differ in at most two grid cells. This construction applies regardless of the convention at the endpoints of $I$.

Monotonicity now gives, uniformly over $s$, $\gamma$, $y$, and $I$,
\begin{align}
\left|P_{N,s,\gamma,y}(I) - m_{\mathbb{T}}(I)\right|
& \leq E_N(q) + \frac{2}{q}.
\end{align}
Taking the supremum over these parameters, then letting $N \rightarrow \infty$ and subsequently $q \rightarrow \infty$, proves the interval assertion.

Finally, let $F : \mathbb{T} \rightarrow \R$ be Riemann-integrable. For every $\eta > 0$, there exist real-valued step functions $F^-$ and $F^+$, constant on the intervals of a finite partition of $\mathbb{T}$, such that $F^- \leq F \leq F^+$ and
\begin{align}
\int_{\mathbb{T}} \big(F^+ - F^-\big) \, dm_{\mathbb{T}}
& \leq \eta.
\end{align}
By the uniform interval estimate just proved and linearity, the averages of $F^-$ and $F^+$ converge to their respective integrals uniformly over $s$, $\gamma$, and $y$. The pointwise inequalities therefore imply
\begin{align}
\limsup_{N \rightarrow \infty} \sup_{\substack{\gamma_N \leq \gamma \leq 2\pi \\ s \in [0, 2\pi) \\ y \in \mathbb{T}}} \left|\E_{n \in K_{s,\gamma}(N)} F\big(y + \Omega(n) \alpha\big) - \int_{\mathbb{T}} F \, dm_{\mathbb{T}}\right|
& \leq \eta.
\end{align}
Letting $\eta \rightarrow 0$ proves the assertion for real-valued $F$. Applying this argument to the real and imaginary parts gives the conclusion for complex-valued Riemann-integrable functions.
\end{proof}

\subsection{Angular-Dynamical Asymptotic Independence}
Theorem \ref{thm:additivesec} and Lemma \ref{lem:countsec} imply asymptotic independence of the relative angular position $u_{s,\gamma}(n)$ and the dynamical variable $T^{\Omega(n)} x$, uniformly over the sectors and initial points in the theorem.

\begin{corollary}\label{implication2}
Let $X$, $\mu$, $T$, and $(\gamma_N)$ satisfy the hypotheses of Theorem \ref{thm:additivesec}. Then, for every Riemann-integrable function $\varphi : [0, 1] \rightarrow \mathbb{C}$ and every $f \in C(X)$,
\begin{align}
\sup_{\substack{\gamma_N \leq \gamma \leq 2\pi \\ s \in [0, 2\pi) \\ x \in X}} \left|\E_{n \in K_{s,\gamma}(N)} \varphi(u_{s,\gamma}(n)) f(T^{\Omega(n)} x) - \left(\int_0^1 \varphi(t) \, dt\right) \left(\int_X f \, d\mu\right)\right|
& \longrightarrow 0
\label{eq:angular-dynamical-rescaled}
\end{align}
as $N \rightarrow \infty$. Moreover, for every $f \in C(X)$, the convergence is uniform over indicator weights of relative angular intervals:
\begin{align}
\sup_{\substack{\gamma_N \leq \gamma \leq 2\pi \\ s \in [0, 2\pi) \\ x \in X \\ 0 \leq a < b \leq 1}} \left|\E_{n \in K_{s,\gamma}(N)} \one_{[a,b)}(u_{s,\gamma}(n)) f(T^{\Omega(n)} x) - (b - a) \int_X f \, d\mu\right|
& \longrightarrow 0.
\label{eq:angular-dynamical-intervals}
\end{align}
\end{corollary}

\begin{proof}
Fix $f \in C(X)$, and set $L := \int_X f \, d\mu$. We first prove \eqref{eq:angular-dynamical-intervals}, writing $D_N(f)$ for the supremum appearing there.

For $0 \leq a < b \leq 1$, let $s_a \in [0, 2\pi)$ be the representative of $s + a \gamma$ modulo $2\pi$, and set
\begin{align}
K_{s,\gamma}^{a,b}(N)
& := \{n \in K_{s,\gamma}(N) : u_{s,\gamma}(n) \in [a,b)\}
= K_{s_a,(b - a) \gamma}(N).
\end{align}
By Lemma \ref{lem:countsec},
\begin{align}
|K_{s,\gamma}(N)|
& = \frac{\gamma N}{2} + O(N^{1/2}), &
|K_{s,\gamma}^{a,b}(N)|
& = \frac{(b - a) \gamma N}{2} + O(N^{1/2}),
\end{align}
uniformly over $s$, $\gamma$, $a$, and $b$. Since $\gamma_N N^{1/2} \rightarrow \infty$, the first estimate gives $|K_{s,\gamma}(N)| \geq \gamma N/4 > 0$ for all sufficiently large $N$, uniformly over the admissible sectors. Consequently,
\begin{align}
\frac{|K_{s,\gamma}^{a,b}(N)|}{|K_{s,\gamma}(N)|}
& = b - a + O\left(\frac{1}{\gamma_N N^{1/2}}\right),
\label{eq:relative-subsector-proportion}
\end{align}
uniformly over all the same parameters.

Fix $\delta \in (0, 1)$. The sequence $(\delta \gamma_N)$ satisfies the same hypotheses as $(\gamma_N)$ in Theorem \ref{thm:additivesec}. Whenever $b - a \geq \delta$, the subsector $K_{s,\gamma}^{a,b}(N)$ has width at least $\delta \gamma_N$. These subsectors are therefore nonempty for all sufficiently large $N$, uniformly over the parameters in this range, and Theorem \ref{thm:additivesec} gives
\begin{align}
\sup_{\substack{\gamma_N \leq \gamma \leq 2\pi \\ s \in [0, 2\pi) \\ x \in X \\ 0 \leq a < b \leq 1 \\ b - a \geq \delta}} \left|\E_{n \in K_{s,\gamma}^{a,b}(N)} f(T^{\Omega(n)} x) - L\right|
& \longrightarrow 0.
\end{align}
Using the identity
\begin{align}
\E_{n \in K_{s,\gamma}(N)} \one_{[a,b)}(u_{s,\gamma}(n)) f(T^{\Omega(n)} x)
& = \frac{|K_{s,\gamma}^{a,b}(N)|}{|K_{s,\gamma}(N)|} \E_{n \in K_{s,\gamma}^{a,b}(N)} f(T^{\Omega(n)} x),
\end{align}
together with \eqref{eq:relative-subsector-proportion} and the fact that the cardinality ratio is at most $1$, proves \eqref{eq:angular-dynamical-intervals} with the supremum restricted to $b - a \geq \delta$.

For the remaining intervals, where $b - a < \delta$, we do not need to average over the subsector, which may be empty. Instead, \eqref{eq:relative-subsector-proportion} gives
\begin{align}
\left|\E_{n \in K_{s,\gamma}(N)} \one_{[a,b)}(u_{s,\gamma}(n)) f(T^{\Omega(n)} x) - (b - a) L\right|
& \leq \|f\|_\infty \frac{|K_{s,\gamma}^{a,b}(N)|}{|K_{s,\gamma}(N)|} + (b - a) |L| \\
& \leq 2 \delta \|f\|_\infty + O\left(\frac{\|f\|_\infty}{\gamma_N N^{1/2}}\right),
\end{align}
uniformly over all the parameters. Combining the two ranges yields
\begin{align}
\limsup_{N \rightarrow \infty} D_N(f)
& \leq 2 \delta \|f\|_\infty.
\end{align}
Letting $\delta \rightarrow 0$ proves \eqref{eq:angular-dynamical-intervals}.
By linearity, \eqref{eq:angular-dynamical-rescaled} holds for step functions on finite partitions of $[0, 1)$ into half-open intervals. For a Riemann-integrable $\varphi$ and $\varepsilon > 0$, choose step functions $v, w$ on $[0, 1)$ such that $w \geq 0$, $|\varphi - v| \leq w$, and $\int_0^1 w(t) \, dt \leq \varepsilon$. Apply the step-function conclusion to $v$ with $f$ and to $w$ with the constant function $1$. The pointwise bound then shows that the limsup of the supremum in \eqref{eq:angular-dynamical-rescaled} is at most $(\|f\|_\infty + |L|) \int_0^1 w(t) \, dt \leq 2 \varepsilon \|f\|_\infty$. Letting $\varepsilon \rightarrow 0$ proves the assertion.
\end{proof}

In particular, for every Riemann-integrable function $\psi : \mathbb{T}_{\mathrm{ang}} \rightarrow \mathbb{C}$, taking $s = 0$, $\gamma = 2\pi$, and $\varphi(t) := \psi(2\pi t + 2\pi\Z)$ in \eqref{eq:angular-dynamical-rescaled} gives, for every $f \in C(X)$,
\begin{align}
\sup_{x \in X} \left|\E_{n \in B_N} \psi(\arg(n)) f(T^{\Omega(n)} x) - \left(\int_{\mathbb{T}_{\mathrm{ang}}} \psi \, dm_{\mathrm{ang}}\right) \left(\int_X f \, d\mu\right)\right|
& \longrightarrow 0.
\end{align}
Thus the full-disk angular-dynamical independence statement also holds uniformly in the initial point.

\subsection{Sectorial Equidistribution in Finite Groups}
The next corollary gives an arithmetic criterion for equidistribution of finite-group-valued completely multiplicative functions in logarithmically shrinking sectors.

\begin{corollary}\label{implication3}
Let $G$ be a finite abelian group, written multiplicatively, and let $\chi : \Z[i] \setminus \{0\} \rightarrow G$ be completely multiplicative. For each $g \in G$, set $P_g := \{p \in \mathcal{P}_{\Z[i]} : \chi(p) = g\}$, and suppose that the elements $g \in G$ satisfying $\sum_{p \in P_g} \mathcal{N}(p)^{-1} = \infty$ generate $G$.
Set $d := \#\{\chi(p) : p \in \mathcal{P}_{\Z[i]}\}$, and let $A_d > 0$ be the constant supplied by Theorem \ref{thm:multisec} for this value of $d$. Then
\begin{align}
\sup_{\substack{(\log N)^{-A_d} \leq \gamma \leq 2\pi \\ s \in [0, 2\pi)}} \max_{a \in G} \left|\frac{\big|K_{s,\gamma}(N) \cap \chi^{-1}(\{a\})\big|}{|K_{s,\gamma}(N)|} - \frac{1}{|G|}\right|
& \longrightarrow 0,
\quad \text{ as } N\to \infty .
\end{align}
\end{corollary}

\begin{proof}
Equip $G$ with the discrete topology and consider the multiplicative action $S_n(y) := \chi(n) y$. This action has exactly $d$ distinct prime transformations, since distinct values of $\chi$ on primes induce distinct translations of $G$. By Lemma \ref{lem:pretendedInvariance}, any probability measure pretending to be invariant must be invariant under translation by every $g \in G$ for which $\sum_{p \in P_g} \mathcal{N}(p)^{-1} = \infty$. These elements generate $G$, so the unique such measure is the uniform measure $m_G$. Thus $(G,S)$ is strongly uniquely ergodic.

Apply Theorem \ref{thm:multisec} with this value of $d$, the identity of $G$ as initial point, and the singleton indicators, each of which has $m_G$-integral $1/|G|$. Taking the maximum over the finitely many elements of $G$ proves the result.
\end{proof}

The finite-group construction also gives joint equidistribution of prime-factor counts associated with a partition of the Gaussian primes into two classes. No angular-distribution hypothesis on either class is required.

\begin{corollary}\label{implication4}
Let $\mathcal{P}_{\Z[i]} = P_1 \sqcup P_2$ be a partition into two sets invariant under multiplication by units, and suppose that
\begin{align}
\sum_{p \in P_j} \frac{1}{\mathcal{N}(p)}
& = \infty, \qquad j \in \{1, 2\}.
\end{align}
For $j \in \{1, 2\}$, let $\Omega_j(n)$ denote the number of prime factors of $n$ belonging to $P_j$, counted with multiplicity, with $\Omega_j(u) := 0$ for every unit $u$. Let $A_2$ be as in Theorem \ref{thm:sector-disk}. Then, for every fixed integer $q \geq 2$, and as $N \rightarrow \infty$
\begin{align}
\sup_{\substack{(\log N)^{-A_2} \leq \gamma \leq 2\pi 
    \\ s \in [0, 2\pi)}} 
\max_{\substack{a,b\in\mathbb N_0 \\ a,b<q}} 
\left|\frac{\big|\{n \in K_{s,\gamma}(N) : (\Omega_1(n), \Omega_2(n))
\equiv (a, b) \pmod q\}\big|}{|K_{s,\gamma}(N)|} - \frac{1}{q^2}\right|
& \longrightarrow 0.
\end{align}
\end{corollary}

\begin{proof}
Invariance of $P_1$ and $P_2$ under multiplication by units makes the counts $\Omega_1$ and $\Omega_2$ well-defined. Unique factorization shows that both counts are completely additive. Equip $G_q := (\Z/q\Z)^2$ with the discrete topology, and define a multiplicative action on $G_q$ by
\begin{align}
S_n(x_1, x_2)
& := \big(x_1 + \Omega_1(n), x_2 + \Omega_2(n)\big),
\end{align}
where both coordinates are taken modulo $q$.

For $p \in P_1$, the transformation $S_p$ is translation by $(1, 0)$, while for $p \in P_2$ it is translation by $(0, 1)$. Thus the action has exactly two distinct prime transformations. By Lemma \ref{lem:pretendedInvariance}, divergence of the reciprocal-norm sums over both $P_1$ and $P_2$ implies that every probability measure that pretends to be invariant is invariant under both coordinate translations. These generate all translations of $G_q$, so such a measure must be uniform. Conversely, the uniform probability measure is invariant under every translation. Hence this system is strongly uniquely ergodic.

Apply Theorem \ref{thm:multisec} with $d = 2$, initial point $(0, 0)$, and the indicators of the singleton subsets of $G_q$. Each indicator has integral $1/q^2$ against the uniform measure. Since $q$ is fixed, taking the maximum over these finitely many indicators proves the desired result.
\end{proof}

\section{Concluding Remarks}
\noindent\textbf{Polynomially shrinking sectors}\\
Theorem \ref{thm:additivesec} allows every subpolynomial shrinking sequence, but does not reach widths $N^{-\delta}$ for any fixed $\delta > 0$. A natural next step is to obtain convergence in some polynomial shrinking range. More ambitiously, Lemma \ref{lem:countsec} gives a relative lattice-counting error of $O(N^{-\varepsilon})$ uniformly over sectors of width at least $N^{-1/2 + \varepsilon}$, suggesting the following question.

\begin{question}\label{qu:polynomialSectors}
Let $X$ be a compact metric space, let $T : X \rightarrow X$ be continuous, and suppose that $\mu$ is the unique $T$-invariant Borel probability measure on $X$. Is it true that, for every fixed $\varepsilon \in (0, 1/2)$ and every $f \in C(X)$,
\begin{align}
\sup_{\substack{N^{-1/2 + \varepsilon} \leq \gamma \leq 2\pi \\ s \in [0, 2\pi) \\ x \in X}} \left|\E_{n \in K_{s,\gamma}(N)} f(T^{\Omega(n)} x) - \int_X f \, d\mu\right|
& \longrightarrow 0,
\qquad \text{ as } N\to\infty?
\end{align}
\end{question}

At the critical scale $N^{-1/2}$, sectors can isolate a single lattice row, bringing one-variable polynomial arithmetic into the problem. For an integer $R \geq 2$, set $N_R := R^2 + 1$, $s_R := \arctan(1/(2R))$, and $\beta_R := \arctan(2/R) - s_R$. Then
\begin{align}
K_{s_R,\beta_R}(N_R)
& = \{m + i : m \in \Z,\ R/2
< m
\leq R\}, \qquad \beta_R
\sim \frac{3}{2R}.
\label{eq:critical-sector-row}
\end{align}
Indeed, a point $m + bi$ in this sector has positive integer coordinates. The norm bound gives $m \leq R$, while the upper angular bound gives $m > Rb/2$, forcing $b = 1$. The remaining conditions give precisely the interval in \eqref{eq:critical-sector-row}.

Let $\lambda_{\Z}$ denote the ordinary Liouville function on the positive integers. Every Gaussian prime dividing $m + i$ has norm equal to a rational prime: an inert rational prime cannot divide its imaginary part. Multiplicativity of the norm therefore gives $(-1)^{\Omega(m + i)} = \lambda_{\Z}(m^2 + 1).$ 
Consequently, an extension of parity cancellation to the sectors in \eqref{eq:critical-sector-row} would imply vanishing Liouville averages along $m^2 + 1$ over $R/2 < m \leq R$, and hence, by dyadic decomposition,
\begin{align}
\frac{1}{R} \sum_{1 \leq m \leq R} \lambda_{\Z}(m^2 + 1)
& \longrightarrow 0.
\end{align}
This is the $m^2 + 1$ case of Chowla's conjecture for polynomial values of the Liouville function; see \cite{Teravainen2024}*{Section 1} for its distinction from the weaker problem of infinitely many sign changes.
This observation concerns the critical scale itself, not the wider sectors in Question \ref{qu:polynomialSectors}.
\bigskip 

\noindent\textbf{The range of the sector-disk comparison}\\
The logarithmic shrinking range in Theorem \ref{thm:sector-disk} is narrower than the subpolynomial range of Theorem \ref{thm:additivesec}. It is natural to ask whether the comparison remains valid in the latter range, retaining uniformity in the initial point and imposing no ergodicity assumption.

\begin{question}
Let $(Y, S)$ be a finitely generated multiplicative topological dynamical system, and let $(\gamma_N)_{N \geq 3}$ be a sequence of positive real numbers satisfying $\gamma_N \rightarrow 0$ and $\gamma_N^{-1} = N^{o(1)}$. Is it true that, for every $g \in C(Y)$,
\begin{align}
\sup_{\substack{\gamma_N \leq \gamma \leq 2\pi \\ s \in [0, 2\pi) \\ y \in Y}} \left|\E_{n \in K_{s,\gamma}(N)} g(S_n y) - \E_{n \in B_N} g(S_n y)\right|
& \longrightarrow 0,
\qquad \text{ as } N\to\infty?
\end{align}
\end{question}

The issue is whether assigning finitely many transformations to primes can preserve angular bias on scales where Theorem \ref{thm:additivesec} already gives convergence for the single count $\Omega$. In the present proof, normalization by the sector cardinality introduces a factor $\gamma^{-1}$ into the scalar comparison, and the available power-of-logarithm cancellation does not settle this question. A positive answer would hold without any assumption on the angular distribution of the sets of primes inducing the individual transformations. Under strong unique ergodicity, it would also extend Theorem \ref{thm:multisec} to the subpolynomial shrinking range, with convergence to the invariant integral still asserted for each fixed initial point.
\bigskip 

\noindent\textbf{Other number fields.}\\
The norm-truncated ideal averages studied by C\'{e}spedes and Donoso \cite{CD26} are defined over arbitrary number fields, where unique factorization into prime ideals is always available. Extending our angularly restricted element averages requires additional geometric choices. Outside $\mathbb{Q}$ and the imaginary quadratic fields, the unit group is infinite, so an absolute field-norm bound alone does not give finite averaging sets. One must therefore impose additional archimedean restrictions or choose representatives modulo units. Moreover, multiplication in the Minkowski embedding is no longer described by a single planar rotation and dilation.

Imaginary quadratic fields of class number one provide a natural setting for extending the additive argument: the planar geometry and unique factorization persist, and Euclideanity is not required. The shrinking-sector prime-semiprime construction would nevertheless require uniform prime counts in short annular sectors analogous to Theorem \ref{thm:stucky}. For larger class number, one may still define $\Omega(\alpha)$ through the prime-ideal factorization of $(\alpha)$, but principal ideals can have nonprincipal prime factors. An adaptation must therefore account for the class-group constraints rather than assume factorization into prime elements.

For multiplicative actions, Theorem \ref{thm:sector-disk} separates the quantitative comparison of averaging regions from the identification of the limiting integral. Extending the comparison would require an analogue of Proposition \ref{prop:angularcancellation} for the chosen geometry, together with an appropriate replacement for the prime-factor data used in its proof. The ergodic conclusion would additionally require a convergence theorem for the unrestricted averages, playing the role of \cite{DLMS2024}*{Theorem E}. We do not pursue these extensions here.

\medskip
\textbf{AI Disclosure}: In preparing this manuscript, the authors used ChatGPT-Sol Pro to assist with language editing, syntax, and proofreading of author-written drafts, and to identify potentially relevant literature. The mathematical arguments and results are the authors' own. All AI-generated suggestions were reviewed by the authors, and all literature references identified with AI assistance were independently verified. The authors take full responsibility for the content and accuracy of the manuscript.

\end{document}